\documentclass[hidelinks,onefignum,onetabnum]{siamart250211}

\usepackage{amsfonts,amssymb,amsmath,nicefrac}
\usepackage{enumitem}
\usepackage{bm}
\usepackage{xcolor,nicefrac,colonequals}
\usepackage{footnote}
\usepackage{cleveref}

\setlist[enumerate,1]{label=(\roman*)}

\usepackage{mathtools}

\DeclareMathOperator*{\dive}{div}

\newcommand{\ua}{{u_\mathsf a}}
\newcommand{\ub}{{u_\mathsf b}}

\newcommand{\rmd}{{\mathrm d}}
\newcommand{\rmy}{{\mathrm y}}

\newcommand{\bx}{{\bm x}}
\newcommand{\dt}{{\mathrm dt}}
\newcommand{\dx}{{\mathrm d\bx}}

\newcommand{\Rb}{{\mathbb R}}
\newcommand{\Nb}{{\mathbb N}}
\newcommand{\U}{{\mathcal U}}

\newcommand{\Uad}{{\mathcal U_\mathsf{ad}}}
\newcommand{\Utad}{{\U^\tau_\mathsf{ad}}}
\newcommand{\Uthad}{{\mathcal U_\mathsf{ad}^{\tau+h}}}

\definecolor{darkblue}{rgb}{0.0, 0.0, 0.55}
\definecolor{darkorange}{rgb}{1.0, 0.55, 0.0}
\definecolor{darkolivegreen}{rgb}{0.33, 0.42, 0.18}

\newsiamremark{remark}{Remark}
\newsiamthm{claim}{Claim}
\newsiamthm{assumption}{Assumption}

\usepackage{hyperref}
\crefname{assumption}{\textup{Assumption}}{\textup{Assumptions}}
\creflabelformat{assumption}{#2\textup{#1}#3} %textup also the numbers not only the name
\crefname{theorem}{\textup{Theorem}}{\textup{Theorems}}
\creflabelformat{theorem}{#2\textup{#1}#3} %textup also the numbers not only the name
\crefname{lemma}{\textup{Lemma}}{\textup{Lemmas}}
\creflabelformat{lemma}{#2\textup{#1}#3} %textup also the numbers not only the name
\crefname{proposition}{\textup{Proposition}}{\textup{Propositions}}
\creflabelformat{proposition}{#2\textup{#1}#3} %textup also the numbers not only the name
\crefname{cor}{\textup{Corollary}}{\textup{Corollaries}}
\creflabelformat{cor}{#2\textup{#1}#3} %textup also the numbers not only the name
\crefname{definition}{\textup{Definition}}{\textup{Definitions}}
\creflabelformat{definition}{#2\textup{#1}#3} %textup also the numbers not only the name
\crefname{remark}{\textup{Remark}}{\textup{Remarks}}
\creflabelformat{remark}{#2\textup{#1}#3} %textup also the numbers not only the name
\crefname{figure}{\textup{Figure}}{\textup{Figures}}
\creflabelformat{figure}{#2\textup{#1}#3} %textup also the numbers not only the name
\crefname{table}{\textup{Table}}{\textup{Tables}}
\creflabelformat{table}{#2\textup{#1}#3} %textup also the numbers not only the name
\crefname{algorithm}{\textup{Algorithm}}{\textup{Algorithms}}
\creflabelformat{algorithm}{#2\textup{#1}#3} %textup also the numbers not only the name
\crefname{section}{\textup{Section}}{\textup{Sections}}
\creflabelformat{section}{#2\textup{#1}#3} %textup also the numbers not only the name

\headers{Differentiability of the value function}{A. Dom\'inguez Corella, N. Jork, S. Volkwein}

\title{Differentiability of the value function in control-constrained parabolic problems\thanks{\textbf{Funding.} The first author was supported by the Alexander von Humboldt Foundation.}}

\author{Alberto Dom\'inguez Corella\thanks{Institut de Mathématiques de Jussieu - Paris Rive Gauche, Sorbonne Universit\'e, 75005 Paris, France; \tt \email{alberto.of.sonora@gmail.com}}
	\and Nicolai Jork\thanks{Department of Mathematics, Eberhard-Karls-Universität T\" ubingen, D-72076 T\"ubingen, Germany; \tt \email{nicolai.jork@uni-tuebingen.de}}
	\and Stefan Volkwein\thanks{Department of Mathematics and Statistics, Universität Konstanz, Germany; \tt \email{stefan.volkwein@uni-konstanz.de}}}

\begin{document}
	
	\maketitle
	%\tableofcontents
	
	\begin{abstract}
		Along the optimal trajectory of an optimal control problem constrained by a semilinear parabolic partial differential equation, we prove the differentiability of the value function with respect to the initial condition and, under additional assumptions on the solution of the state equation, {the differentiability of the value function with respect to the initial time}. In our proof, we rely on local growth assumptions commonly associated with the study of second-order sufficient conditions. These assumptions are generally applicable to a wide range of problems, including, for instance, certain tracking-type problems. Finally, we discuss the differentiability of the value function in a neighborhood of the optimal trajectory when a growth condition for optimal controls is used.
	\end{abstract}
	
	% REQUIRED
	
	% REQUIRED

	%%%%%%%%%%%%%
	\section{Introduction and literature review}
	%%%%%%%%%%%%%
	
	%%%%%%%%%%%%%
	\subsection{Introduction}
	%%%%%%%%%%%%%
	In the theory of optimal control, one is interested in the minimization of
	functionals subject to several constraints. Among these, an equation usually encapsulates the corresponding system's dynamics on which the so-called control function acts. This paper considers a particular type of optimal control problem where the system's dynamic is given by a semilinear parabolic partial differential equation and where at most a linear appearance of the control function occurs in the functional and the constraining PDE. For problems of this kind, we are interested in the regularity properties of the corresponding value function. For simplicity, we focus on tracking-type problems. The value function represents the minimum cost achieved when starting at a given initial state and following the trajectory determined by the controlled dynamical system. The study of the regularity of the value function is known to aid the investigation of properties of optimal control problems and their solutions. If the value function satisfies certain regularity assumptions, for instance, continuous differentiability with respect to the initial state {and further regularity with respect to the time variable}, the value function satisfies a Hamilton-Jacobi-Bellman equation (HJB). The continuous differentiability of the value function is a prerequisite for policy iteration and to derive HJB-based feedback controls. For objective functionals where a Tikhonov regularization term is present and under nonlinear ODE-constraints, this was studied in \cite{KK24}. When there does not appear a Tikhonov regularization term, one can not expect the same feedback formulas as in the regularized case and the construction of such laws is technically involved. Thus, while we expect that this work is helpful to the investigation of feedback laws for the optimal solutions, this is kept for future work. Motivated by the above, the present paper aims to obtain regularity properties of the value function corresponding to optimal control problems subject to box constraints. This investigation of regularity properties of this kind of problem is challenging. The first difficulty is due to the semilinear structure of the state equation, which requires heavy use of PDE theory. The affine structure of the optimal control problem portrays the second difficulty. No Tikhonov regularization term is present in the objective functional, and we cannot apply arguments typically used in the standard theory of optimal control problem for PDEs. This primarily means that, unlike the Tikhonov regularized problem, we cannot infer a quantitative Legendre-Clebsch type assumption involving the $L^2$-distance of the controls, considered, for instance, in \cite{KB_2022,KB24}. Instead, we need to work with a recently introduced assumption, \cite{CDJ_2023,DJV_2023}, on the joint growth of the first and second variation of the cost functional, involving the linearized states, which is discussed below in detail. This assumption, crucial in the present paper, naturally displays a weaker assumption than that used for the Tikhonov regularized problem. {The study of optimality conditions for affine optimal control problems with PDE-constraints is preceded by the study of optimality conditions for optimal control problems with ODE-constraints. Indeed, much of the development of implementable second-order sufficient conditions (SOSC) originated in the study of problems governed by ordinary differential equations. In this setting, the control–affine structure allows quadratic growth conditions to be expressed via Riccati equations or Goh-type transformations, yielding implementable versions of the shooting algorithm. Important contributions include the work \cite{zbMATH07006554,zbMATH06104987, zbMATH06656093, zbMATH02139063,zbMATH05700767,zbMATH06135817}. We refer also to the works \cite{ zbMATH07262056,zbMATH07223011} that apply sufficient optimality conditions for the study of stability of the optimal solutions.}
    In our study, the differentiability of the value function is proved with respect to initial data in $L^2(\Omega)$. This implies low regularity of the corresponding optimal states and restricts the model's generality to simpler nonlinear equations as in \cite[Section 6]{KB24}. Using the same arguments in the proofs below, one can consider initial data in $L^\infty(\Omega)$, which allows for more general models. Further, in some instances, where we additionally provide the differentiability in a neighborhood of the initial data, to overcome the problem of low regularity of the states, we consider differentiability with respect to the $ L^\infty$-Norm.
    Our main results are the following.
    \smallskip
    \begin{enumerate}
        \item The continuous differentiability of the value function along the optimal trajectory with respect to the initial data.
        \item Under additional assumptions on the optimal state, the continuous differentiability of the value function along the optimal trajectory with respect to the time parameter.
        \item The continuous differentiability of the value function in a neighborhood along the optimal trajectory with respect to the initial data and the time parameter for a sub-class of problems.
    \end{enumerate} 
    To the authors' best knowledge, it is the first time that the differentiability of the value function is proved for unregularized optimal control problems with PDE constraints under such weak assumptions. Furthermore, the results are new even for optimal control problems with ODE constraints. We end this introduction with a short discussion of related literature.
	%%%%%%%%%%%%%
	\subsection{Literature review}
	%%%%%%%%%%%%%
	Several publications in recent years have considered certain aspects of the value function's regularity for optimal control problems constrained by ODEs or PDEs. The following is a non-exhaustive summary of related results. Since the approaches usually differ when considering finite or infinite horizon problems, we split the literature into these two cases.
	\noindent
	\subsubsection{Finite horizon problems}
	In \cite{CF2013}, it is shown that the value function of a Bolza optimal control problem can fail to be differentiable everywhere. Furthermore, it is shown that if the value function is proximally subdifferentiable at a point {$(t,x)$}, then it is smooth in a neighborhood of that point and that the value function stays smooth in a neighborhood of any optimal trajectory starting at a point where the proximal subdifferential is nonempty. We mention that the objective functional contains a quadratic control term in the problem considered in \cite{CF2013}. In \cite{G2005}, optimal control problems with convex costs, for which Hamiltonians have Lipschitz-continuous gradients, are considered. Such problems include extensions of the linear-quadratic regulator with hard and possibly state-dependent control constraints and linear-quadratic penalties. Lipschitz-continuous differentiability and strong convexity of the terminal cost are shown to be inherited by the value function, leading to Lipschitz continuity of the optimal feedback. The paper \cite{BM_00} investigates the differentiability of a Tikhonov regularized finite horizon optimal control problem. As the regularity of solutions to Hamilton-Jacobi equations is a related topic, we also mention the works \cite{EZ2022,FS2004}.
	
	\smallbreak
	\noindent
	\subsubsection{Infinite horizon problems}
	In \cite{MR3853601}, the authors consider the regularity of the value function for an infinite-horizon optimal control problem subject to an infinite-dimensional state equation where the state and control variables appear bilinearly. They obtain that the value function is infinitely differentiable in the neighborhood of the steady state under a stabilizability assumption. Related questions are discussed in the subsequent publications \cite{MR4026593,MR3985547}. In \cite{KB_2022}, an abstract framework was introduced by which the authors obtain the local continuous differentiability of the value function in $L^2$ associated with optimal problems subject to abstract semilinear parabolic equations with control constraints. Subsequently, in \cite{KB24}, an abstract framework for the study of the continuous differentiability of local value functions with $H^1(\Omega)$ initial data was introduced with application to semilinear parabolic equations and control constraints. Consequently, the authors provide the local well-posedness of the associated HJB equation. In both works, it was used that a Tikhonov term is present in the objective functional, which guarantees a uniform stability estimate for the optimal controls.
	%%%%%%%%%%%%%%%%%
	\section{Control model and overview of main results}
	%%%%%%%%%%%%%%%%%
	
	This section introduces the reference control model under investigation and presents our main results. Since we consider a standard optimal control problem widely studied in the literature, we defer the preliminaries to the next section.
	
	%%%%%%%%%%%%%%%%%
	\subsection{Reference model}
	%%%%%%%%%%%%%%%%%
	
	Consider a bounded domain $\Omega\subset \mathbb R^d$, for $d\in\{1,2,3\}$ with Lipschitz boundary $\partial \Omega$. Let $T>0$ be a positive number representing the time horizon. We set $Q\colonequals(0,T)\times\Omega$ and $\Sigma\colonequals(0,T)\times \partial \Omega$. For functions $\ua,\ub:Q\to\mathbb R$ representing control bounds, the set of admissible controls is given by 
	\begin{align*}
		\Uad\colonequals\big\{ u\in L^1(Q)\,\big|\,\ua\le u\le\ub\text{ a.e. in }Q\big\}.
	\end{align*}
	Under proper assumptions (see \cref{Assumption} below) it is well-known that for each control $u\in\Uad$ there is an associated state $y_{u}\in  W(0,T)$ satisfying the parabolic constraint
	\begin{equation}
		\label{system}
		\left\{
		\begin{aligned}
			&\frac{\partial y_u}{\partial t}-\dive\big(A(\bx)\nabla y_u\big) + f(y_u)=u\text{ a.e. in } Q, \\&
			y_u=0\text{ a.e. on }\Sigma,\quad y_u(0)=y_0\text{ a.e. in }\Omega.
		\end{aligned}
		\right.
	\end{equation}
	Here, $y_0\in C(\bar \Omega)$ is a prescribed initial datum, and $A\in L^\infty\big(\Omega; \mathbb R^{d\times d}\big)$  a diffusion matrix function. The nonlinearity in \cref{system} is represented by a function $f:\mathbb R\to\mathbb R$. Details on the functional spaces involved and the weak formulation of equation (\ref{system}) will be given in the next section.
	\smallbreak 
	For a target datum $y_{Q}:Q\to\mathbb R$, the objective functional $\mathcal J:\Uad\to \mathbb R$ is given by 
	\begin{align}\label{objfun}
		\mathcal J(u)\colonequals\frac{1}{2}\int_Q|y_{u}(t,\bx) - y_{Q}(t,\bx)|^2\, \mathrm d\bx\mathrm dt.
	\end{align}
	The optimal control problem we are interested in reads as 
	\begin{align}
		\label{P}
		\tag{\textbf P}
		\min\mathcal J(u)\quad\text{subject to}\,\,\,\, u\in\Uad.
	\end{align}
	The problem is to find a control within the control constraints that makes the state resemble the experimental data. This problem can be seen as an inverse problem and is usually referred to in the literature as a \textit{tracking problem}.
	\smallbreak 
	We consider problem \eqref{P} under the following assumption.
	
	\begin{assumption}
		\label{Assumption}
		The following statements hold.
		\begin{itemize}
			\item[\em(i)] The domain $\Omega$ is bounded and has Lipschitz boundary. The matrix function $A$ belongs to $L^\infty(\Omega; \Rb^{d\times d})$ and is uniformly positive definite and symmetric. The initial datum $y_0$ belongs to $C(\bar \Omega)$. The functions $\ua,\ub$ belong to $L^\infty(Q)$.
			\item[\em(ii)] The function $f:\Rb\to\Rb$ is of class $C^2$, its derivatives $f':\Rb\to\Rb$ and $f'':\Rb\to\Rb$ are bounded and  $	f'(y) \geq 0$ for all $y \in\Rb.$ 
			\item [\em(iii)] The target datum $y_{Q}$ belongs to $L^r(Q)$ for some   $r>1+\nicefrac{d}{2}$.		
		\end{itemize}
	\end{assumption}
	\iffalse 
	\begin{remark}
		\begin{itemize}
			\item [(i)] \cref{Assumption} ensures that equation \cref{system} is well-defined and that problem \cref{P} has at least one global minimizer.
			
			\item [(ii)] It is worth noting that under this assumption, nonlinearities such as $f(\rmy) = \sin\rmy+\rmy$ are permitted in equation \cref{system}, as well as linear functions $f(\rmy)=\alpha\rmy$ for a constant $\alpha > 0$. Moreover, the presented results can be extended in a straightforward way to the case that $f$ also depends on $t$ and $\bx$. To simplify the presentation, we have omitted the extension.
		\end{itemize}
	\end{remark}
	\fi 
	It is worth noting that under this assumption, nonlinearities such as \( f(y) = \sin y + y \) are permitted in equation \cref{system}, as well as linear functions $f(y) = \alpha y$ for some $\alpha > 0$. The presented results can be extended in a straightforward way to the case that $f$ also depends on $t$ and $\bx$. To simplify the presentation, we have omitted this extension.
	\smallbreak
	\cref{Assumption}, {due to \cref{Prop3.1} and \cref{prop3.3}} (below in the preliminaries section), ensures that equation \cref{system} is well-defined and that problem \cref{P} has at least one global minimizer.
	To simplify notation, from now on, we fix a reference global minimizer $\bar u\in\Uad$ of problem \cref{P}. 
	\smallbreak 
	Let us now recall quickly that for each $u\in\Uad$, one can associate an adjoint state $p_{u}\in W(0,T)$ such that the first variation of the objective functional can be identified with it. Indeed, it can be easily checked that at a control $u\in\Uad$
	\begin{align*}
		\mathcal J'(u)v = \int_Q p_{u}(t,\bx) v(t,\bx)\,\mathrm d\bx\dt \quad\text{for all }v\in \Uad -  u,
	\end{align*}
	where $\Uad-u=\{v\,|\,u+v\in\Uad\}$. The first-order necessary optimality condition for the reference minimizer reads as
	\begin{align*}
		p_{\bar u}(u-\bar u)\ge 0\quad\text{a.e. in } Q\text{ and for all }u\in\Uad.
	\end{align*}
	\smallbreak 
	To describe the second variation at the reference minimizer, we introduce a family of functions representing the linearization of equation \cref{system}. For each $v\in \Uad-\bar u$, we denote by $z_v\in W(0,T)$ the unique function satisfying 
	\begin{equation}\label{linstate}
		\begin{cases}
			\displaystyle\frac{\partial z_v}{\partial t} 	- \dive\big(A(\bx)\nabla z_v\big) + f'(y_{\bar u}) z_v = v \text{ a.e. in } Q, \\
			z_v=0 \text{ a.e. on } \Sigma,\  z_v(0) = 0 \text{ a.e. in } \Omega. 
		\end{cases}
	\end{equation}
	For each $v\in \Uad-\bar u$, the second variation can be easily calculated as
	\begin{align*}
		\mathcal J''(\bar u)(v,v) =  \int_Q \Big(1 - p_{\bar u}(t,\bx)f''(y_{\bar u}(t,\bx))\Big)z_{v}^2(t,\bx)\,\dx\dt.
	\end{align*}
	Due to the structural form of the second variation, it is natural to consider second-order conditions involving the norm \( \|z_v\|_{L^2(Q)} \) rather than \( \|v\|_{L^p(Q)}\) for some $p\in[1,+\infty)$. This has been previously noted, and assumptions of the form $\mathcal{J}''(\bar{u})(v,v)\ge c\, \|z_v\|_{L^2(Q)}^2$ for some constant $c>0$, over a cone of critical directions, have been used in the literature as a sufficient condition for strict local optimality; see, e.g., \cite[Theorem 3.1]{CM_2020}. In \cite{CDJ_2023,DJV_2023}, the authors introduced an assumption involving also the first variation of the objective functional. In the sequel, we employ a slightly different assumption, including the factor $\nicefrac{1}{2}$ next to the second variation.
	
	\begin{assumption}
		\label{MA}
		There exist numbers  $\delta>0$ and $c>0$ such that 
		\begin{align}\label{ssc}
			\mathcal J'(\bar u)v+\frac{1}{2}\, \mathcal J''(\bar u)(v,v)\ge c\,{\| z_{v}\|}_{L^2(Q)}^2
		\end{align}
		for all $v\in\Uad-\bar u$ satisfying $\| z_{v} \|_{L^2(Q)}\le\delta$.
	\end{assumption}
	
	The previous assumption can be seen as both necessary and sufficient for addressing stability in optimal control problems. {The sufficiency of \cref{MA} for local optimality and stability for the optimal states has been demonstrated in \cite{CDJ_2023, DJV_2023}, while its necessity for the stability of optimal solutions with respect to perturbations of the tracking data was established in \cite[Theorem 5.3]{DL_2024}}.
	\smallbreak 
	To conclude this subsection, we state a result linking \cref{MA} to a local error bound, which, in turn, ensures strict local optimality.
	
	\begin{proposition}\label{Proploc}
		Suppose that  $\bar u\in\Uad$ satisfies \cref{MA}.  Then, there exist numbers $\delta>0$ and $c>0$ such that
		\begin{align}
			\label{errorbound}
			\mathcal J(u) - \mathcal J(\bar u) \ge c\,{\| y_{u} - y_{\bar u}\|}_{L^2(Q)}^2
		\end{align}
		for all $u\in\Uad$ such that $\|y_{u}-y_{\bar u}\|_{L^2(Q)}\le\delta$.
	\end{proposition}
	
	\begin{proof}
		This follows from a standard application of the Taylor's theorem and estimates involving the objective functional, see, e.g., \cite[Lemma 6]{CT_2016}. 
	\end{proof}
	
	%%%%%%%%%%
	\subsection{A family of problems}
	%%%%%%%%%%
	
	We introduce a family of optimal control problems in which we can embed the problem \eqref{P}. For each \emph{initial time} $\tau\in[0,T]$ we define $Q_\tau\colonequals(\tau,T)\times\Omega$, $\Sigma_\tau\colonequals(\tau,T)\times\partial\Omega$ and set 
	\begin{align*}
		\Utad\colonequals\big\{u\in L^1\big(Q_\tau)\,\big|\,\ua\le u\le\ub\text{ a.e. in }Q_\tau\big\}.
	\end{align*}
	For an \emph{initial condition} $\eta\in L^2(\Omega)$ and an admissible control $u\in\Utad$ there is an associated state $y^{\tau,\eta}_{u}\in W(\tau,T)$ that is the weak solution of equation 
	\begin{align}\label{eq:taustate}
		\left\{
		\begin{aligned}
			&\frac{\partial y}{\partial t}- \dive\big(A(\bx)\nabla y\big) + f(y) = u \text{ a.e. in } Q_\tau, \\
			&y=0 \text{ a.e. on } \Sigma_\tau,\quad  y(\tau) = \eta \text{ a.e. in } \Omega. 
		\end{aligned}
		\right.
	\end{align}
	%
	\iffalse
	It is not necessary to give the weak formulation here. It will be given in the preliminaries section. \textcolor{brown}{Let us recall that by a weak solution to \eqref{equation:1} we understand the following space-time variational formulation
		%
		\begin{align*}
			&\int_t^T{\langle y_t(s),\phi(s)\rangle}_{V',V}+{\langle y(s),\phi(s)\rangle}_V+{\langle f(t,y(s)), \phi\rangle}_H\,\mathrm ds\\
			&=\int_t^T{\langle u(s),\phi(s)\rangle}_H\,\mathrm ds\quad\text{for all }\phi\in L^2(t,T;V)
		\end{align*}
		%
		and with $y$ satisfying the initial condition $y(t,\cdot)=y_0$. It is well known that this variational formulation is equivalent to the pointwise in-time formulation
		%
		\begin{align}
			\label{pointwiseweak}
			&{\langle y_t(s),\varphi\rangle}_{V',V}+{\langle y(s),\varphi\rangle}_V+{\langle f(s,y(s)),\varphi\rangle}_H\\
			&={\langle u(s),\varphi\rangle}_H\quad\text{for all }\varphi\in V
		\end{align}
		%
		and for almost all (f.a.a.) $s\in [t,T]$. From the pointwise in time formulation \eqref{pointwiseweak} we conclude that a solution $y$ to \eqref{equation:1} satisfies for all $\tau\in [t,T]$
		%
		\begin{align*}
			&\int_\tau^T{\langle y_t(s),\phi(s)\rangle}_{V',V}+{\langle y(s),\phi(s)\rangle}_V+{\langle f(s,y(s)), \phi(s)\rangle}_H\,\mathrm ds\\
			&=\int_\tau^T{\langle u(s),\phi(s)\rangle}_H\,\mathrm ds\quad\text{for all }\phi\in L^2(\tau,T;V).
	\end{align*}} 
	\fi 
	%
	\smallbreak
	We consider the family of objective functionals $\mathcal J_{\tau, \eta}:\Utad\to\Rb$ given by
	\begin{align*}
		\mathcal J_{\tau,\eta}(u)\colonequals \frac{1}{2}\int_{Q_\tau} |y^{\tau,\eta}_{u}(t,\bx) - y_{Q}(t,\bx)|^2\,\dx\dt.
	\end{align*}
	The family $\{\mathcal J_{\tau,\eta}\,|\,\tau\in [0,T]\text{ and }\eta\in L^2(\Omega)\}$ of objective functionals induces the family of optimal control problems
	\begin{align}
		\label{Pteta}
		\tag{\text{$\mathbf P_{\tau,\eta}$}}
		\min \mathcal J_{\tau,\eta}(u)\quad\text{subject to}\quad u\in\Utad.
	\end{align}
	Notice that \eqref{P} is the same problem as \eqref{Pteta} for the choices $\tau=0$ and $\eta=y_0$. In this sense, \eqref{P} is an element of the family $\{\eqref{Pteta}\mid\tau\in[0,T]\text{ and }\eta\in L^2(\Omega)\}$.
	\smallbreak 
	For the reference global solution $\bar u$ of \cref{P}, the so-called \textit{Bellman's Optimality Principle} ensures that $\bar u|_{[\tau, T]}$ is a global minimizer of problem  \cref{Pteta} with $\eta = y_{\bar u}(\tau)$; in some sense \textit{optimality propagates forward in time}.
	\subsection{Overview of results}
	%%%%%%%%%%
	
	We now define the main object of interest in this paper, the so-called value function $\upsilon: [0,T] \times L^2(\Omega) \to \Rb$ given by
	\begin{align*}
		\upsilon(\tau, \eta)\colonequals\inf\big\{\mathcal{J}_{\tau,\eta}(u)\,\big|\,u\in\Utad\big\}.
	\end{align*}
	It is well known that under basic assumptions, such as \cref{Assumption}, the value function is locally Lipschitz continuous in the space variable. Thus, by applying Rademacher-type theorems \cite[Theorem 2.5]{P_1990}, one can deduce the existence of Fréchet differentiability points in the space variable. However, this does not allow to  identify the points of differentiability.
    Below, we provide sufficient conditions for a point to be in the set of differentiability points of the value function. For functions $\varphi:(0,T)\times\Omega\to \mathbb{R}$, we adopt a slight abuse of notation by writing its functional lift with the same symbol. In this convention, $\varphi(t)$ is understood as a function from $\Omega$ into $\mathbb{R}$, defined by $\varphi(t)(x) = \varphi(t,x)$ for $x \in \Omega$.
	\iffalse 
	\begin{theorem}\label{thm1}
		The value function $\upsilon:[0,T]\times L^2(\Omega)\to\Rb$ is locally Lipschitz-contin\-uous, i.e., for each $M>0$ there exists a Lipschitz constant $L_{M}>0$ such that 
		%
		\begin{align*}
			\big|\upsilon(\tau_1,\eta_1) - \upsilon(\tau_2,\eta_2)\big|\le L_M \big(|\tau_1-\tau_2| +{\|\eta_1 - \eta_2\|}_{L^2(\Omega)}\big)
		\end{align*}
		%
		for all $\tau_1,\tau_2\in[0,T]$ and $\eta_1,\eta_2\in \mathbb B_{L^2(\Omega)}\big(0,M\big)$. Moreover, the following statements hold.
		%
		\begin{itemize}
			\item[\em(i)] For any $\eta\in L^2(\Omega)$, the function $\tau\mapsto v(\tau,\eta)$ is   a.e. 
			differentiable in $[0,T]$.
			\item[\em(ii)] For any $\tau\in[0,T]$, there exists a dense set $\mathcal D_\tau\subset L^2(\Omega)$ such that the function $\eta \mapsto \upsilon(\tau,\eta)$ is Fr\'echet-differentiable on $\mathcal D_\tau$.
			\item[\em(iii)] There exists a dense set $\mathcal D\subset[0,T]\times L^2(\Omega)$ such that $\upsilon:[0,T]\times L^2(\Omega)\to\Rb$ is Fr\'echet-differentiable on $\mathcal D$.
		\end{itemize} 
	\end{theorem}
	
	\begin{proof}
		The local Lipschitz property of the value function directly stems from the local Lipschitz properties of the problem's data using standard and straightforward arguments:
		Item (i) follows from Rademacher's theorem, and items (ii) and (iii) from the theorem of Preiss \cite[Theorem 2.5]{P_1990}.
	\end{proof}
	\fi 
    \smallbreak 
	Our main result provides sufficient conditions for the $L^2$-differentiability of the value function with respect to the state variable. The  proof is given in Section \ref{proofss}.
	
	\begin{theorem}\label{thm3}
		Suppose $\bar u\in\Uad$ is a unique global minimizer of \eqref{P} satisfying \cref{MA}. For any given  $\tau\in [0,T]$, the function $\upsilon(\tau):L^2(\Omega)\to\mathbb R$ is Fr\'echet-differentiable at $y_{\bar u}(\tau)\in L^2(\Omega)$ and
		\begin{align*}
			\nabla\upsilon\big(\tau,y_{\bar u}(\tau)\big) = p_{\bar u}(\tau)\ \text{ in }L^2(\Omega).
		\end{align*}
	\end{theorem}	
	
	We mention that the proof we give for Theorem~\ref{thm3} allows us to consider differentiability of the value function in the $L^\infty(\Omega)$  sense. In fact, the proof becomes simpler in this case, as the perturbed states exhibit better regularity properties; for further discussion, we refer to  \Cref{ss_diff}.
	\smallbreak 
		\smallbreak
	In general, time differentiability of the value function cannot be expected, as a straightforward application of envelope theorems would show that the derivative of the optimal state appears in the expression for the derivative at differentiability points, and the state does not necessarily have this regularity.

	Using \cref{MA}, we present partial results concerning the existence of one-sided derivatives of the optimal state. The proof is given in Section \ref{proofss}.
	
	\begin{theorem}\label{thm2}
		{Suppose that $\bar u\in\Uad$ is a unique global minimizer of \eqref{P}} satisfying \cref{MA} and suppose that $y_Q\in C\big([0,T], L^2(\Omega)\big)$. 
		The following statements hold.
		\begin{itemize}
			\item[\em(i)] Let $\tau\in[0,T)$. If $\frac{\mathrm d^+y_{\bar u}}{\mathrm dt}(\tau)$ exists in $L^2(\Omega)$, then the function $t\mapsto\upsilon(t,y_{\bar u}(\tau))$ is right differentiable at $\tau$ and its right derivative is given by  
            {
			\begin{align*}
				\frac{\mathrm d^+\upsilon}{\mathrm dt}\big(\tau, y_{\bar u}(\tau)\big)=-\int_\Omega \Big[\frac{1}{2} |y_{\bar u}(\tau,\bx)-y_{Q}(\tau,\bx)|^2+p_{\bar u}(\tau,\bx)\,\frac{\mathrm d^+ y_{\bar u}}{\mathrm dt}(\tau,\bx)\Big]\,\dx.
			\end{align*}
            }
			\item[\em(ii)] Let $\tau\in(0,T]$. If $\frac{\mathrm d^-y_{\bar u}}{\mathrm dt}(\tau)$ exists in $L^2(\Omega)$, then the function $t\mapsto\upsilon(t,y_{\bar u}(\tau))$ is left differentiable at $\tau$ and its left derivative is given by   
			\begin{align*}
				\frac{\mathrm d^-\upsilon}{\mathrm dt}\big(\tau, y_{\bar u}(\tau)\big)=- \int_\Omega \Big[\frac{1}{2}|y_{\bar u}(\tau,\bx)-y_{Q}(\tau,\bx)|^2+p_{\bar u}(\tau,\bx)\,\frac{\mathrm d^- y_{\bar u}}{\mathrm dt}(\tau,\bx)\Big]\,\dx.
			\end{align*}
		\end{itemize}
	\end{theorem}	

	Using \cref{MA} and combining the arguments from the previous results, one can show that the value function is Fréchet-differentiable with respect to both of its variables, provided the optimal state is time-differentiable in $L^2(\Omega)$. The differential can then be computed using the formulas from the previous two results.
	
	\begin{theorem}\label{thm4}
		{Suppose that $\bar u\in\Uad$ is  a unique global minimizer of \eqref{P}}  satisfying \cref{MA} and suppose that $y_Q\in C\big([0,T], L^2(\Omega)\big)$. Let $\tau\in[0,T]$ be such that $y_{\bar u}:[0,T]\to L^2(\Omega)$ is differentiable at $\tau$. Then the value function $v:[0,T]\times L^2(\Omega)\to\Rb$ is differentiable at $(\tau, y_{\bar u}(\tau))$. 
	\end{theorem}	

    \begin{proof}
        	The result follows by the same arguments as in the proofs of \cref{thm3} and \cref{thm2}. 
    \end{proof}
	
	\begin{remark} 
	    \begin{enumerate}
	        \item [(i)] With additional assumptions, the differentiability can be extended to a neighborhood of the graph of the trajectory; we discuss this in more detail in \Cref{diffneigh}.
            \item [(ii)] Differentiability of the value function is, for example, an essential prerequisite for the convergence of policy iteration, which is used to calculate a numerical solution to the HJB equation. In particular, the article \cite{KK24} also considers control constraints, as we do. If \cref{thm3,thm2,thm4} apply, then after a suitable discretization of \eqref{P}, the results from \cite{KK24} are applicable.
	    \end{enumerate}
	\end{remark}
	
	\iffalse 
	Finally, the correctness of our main result is evidenced by the following lemma.
	\begin{lemma}
		Let $\bar u\in\U$ be a global minimizer of \eqref{P}. For any $\tau\in [0,T)$, the map $\tau \to \upsilon(\tau,y_{\bar u}(\tau))$ is differentiable and
		\begin{equation*}
			\frac{\partial}{\partial \tau}\upsilon(\tau,y_{\bar u}(\tau))=-\int_\Omega L(\tau,\bx,y_{\bar u}(\tau,\bx) \,\dx.
		\end{equation*}
	\end{lemma}
	
	\begin{proof}The function $\tau \to \int_\Omega L(t,\bx,y_{\bar u }(t,\bx))\, \dx $ is continuous thus every $\tau\in (0,T)$ is a Lebesgue point.
		Due to the dynamic programming principle, we find
		\begin{align*}
			& \upsilon(\tau+h,y_{\bar u}(\tau+h))-\upsilon(\tau,y_{\bar u}(\tau))=\int_{\tau+h}^T\int_\Omega L(t,\bx,y_{\bar u }(t,\bx))\, \dx \dt\\
			&-\int_{\tau}^T\int_\Omega L(t,\bx,y_{\bar u }(t,\bx))\, \dx \dt=-\int_{\tau}^{\tau+h}\int_\Omega L(t,\bx,y_{\bar u }(t,\bx))\, \dx \dt.
		\end{align*}
		Dividing by $h$ and taking the limit gives the claim.
	\end{proof}
	\fi 
	
	%%%%%%%%
	\section{Preliminaries} 
	%%%%%%%%
	
	%%%%%%%%
	\subsection{Well-posedness of the  family of  problems}
	%%%%%%%%
	
	In this subsection, we provide a precise formulation of the assumptions underlying the reference control problem introduced in the previous section. 
	
	%%%%%%%%
	\subsubsection*{Some functional spaces}
	%%%%%%%%
	
	Recall that we are considering a bounded domain $\Omega\subset\Rb^d$ for $d\in\{1,2,3\}$ with Lipschitz boundary $\partial\Omega$ and a time horizon $T>0$. We denote by $C^{\infty}_0(\Omega)$ the space of smooth functions with compact support; its topological dual is the space of distributions. The Sobolev spaces $W_0^{s,p}(\Omega)$ have their usual meaning.  Following  classic conventions, we denote $H^1_0(\Omega)\colonequals W_{0}^{1,2}(\Omega)$ and $H^{-1}(\Omega)\colonequals H^1_0(\Omega)^*$. Let $I\subseteq [0,T]$ be a closed interval. {Given a Banach space $X$}, $L^p(I;X)$ is the usual space of $p$-integrable functions in the Bochner sense, and $C(I;X)$ is the space of continuous functions from the closed interval $I$ into $X$. The primary function space in the weak formulation of parabolic PDEs over $I$ is the space $W(I)$; it consists of functions $y:I\to H^{-1}(\Omega)$ such that 
    {
	\begin{itemize}
		\item[(i)] $\frac{\mathrm dy}{\mathrm dt}(t)\in H^{-1}(\Omega)$ exists for almost all $t\in I$ and $\int_{I}\|\frac{\mathrm dy}{\mathrm dt}(t)\|^2_{H^{-1}(\Omega)}\,\mathrm dt<+\infty$.
		\item[(ii)] $y(t)\in H^1_0(\Omega)$ for almost all $t\in I$ and $\int_{I}\|y(t)\|^2_{H^{1}_0(\Omega)}\,\mathrm dt<+\infty$.
	\end{itemize}
    }
	It is well-known that $W(I)$ is a Hilbert space with continuous embedding $W(I)\hookrightarrow C(I;L^2(\Omega))$; see, e.g., \cite[Theorem 3.10]{Tro2010}. 
	\smallbreak 
	From now on, we will denote $V\colonequals H_0^1(\Omega)$ and $H\colonequals L^2(\Omega)$. Recall that $(V,H,V^*)$ is a Gelfand triple. The duality pairing of $V$ and $V^*$ is compatible with the inner product on $H$ (which is by Riesz representation the duality pairing between $H$ and $H^*$) in the sense that
	\begin{align*}
		{\langle y,p\rangle}_{V^*,V}  ={\langle y,p\rangle}_{H^*,H} \quad \text{for all $y\in H$ and $p\in V$}.
	\end{align*}	
	We slightly abuse the notation and write simply $\langle \cdot\,,\cdot\rangle$ for both duality pairings. It is known  that for any two elements $y,p\in W(I)$, the integration by parts formula
    {
	\begin{align*}
		\int_0^T \Big\langle \frac{d y}{d t}(t), p(t)\Big\rangle\, \dt = {\big\langle y(T),p(T)\big\rangle}-{\big\langle y(0),p(0)\big\rangle}-\int_0^T\Big\langle \frac{dp}{dt}(t),y(t)\Big\rangle\, \dt
	\end{align*}
    }
	holds, see \cite[Theorem 3.11]{Tro2010} for more details. 
	
	%%%%%%%%	
	\subsubsection*{Weak solutions of  parabolic PDEs and existence of minimizers}
	%%%%%%%%	
	
	Let $\tau\in[0,T]$ and $\eta\in H$ be given. Given $u\in\Utad$, we say that $y_{u}^{\tau,\eta}\in W(I)$ is a weak solution of \eqref{eq:taustate} if $y_u^{\tau,\eta}(\tau)=\eta$ and for any $\varphi\in H^1\big(\tau,T;H^1(\Omega)\big)$ satisfying $\varphi(T)=0$ there holds 
	\begin{align}
		\label{pointwiseweak1}
		\int_{Q_\tau}\Big[y\frac{\mathrm d\varphi}{\mathrm dt} +\langle A\nabla y,\nabla \varphi\rangle+ f(y)\varphi\Big]\,\dx\dt = \int_{Q_\tau} u\varphi\,d\bx\,dt + \int_{\Omega}\eta \varphi(0)\,\dx.
	\end{align}
	We refer to \cite[pp. 266-267]{Tro2010} for more details on weak formulations.   
	
	Under \cref{Assumption}, it is possible to ensure the existence of weak solutions for the {state equation}.

	\begin{proposition}\label{Prop3.1}
		For any $\tau\in[0,T]$ and $u\in\Utad$, equation  \eqref{eq:taustate} possesses a unique weak solution  $y_{u}^{\tau,\eta}\in W(\tau,T)$. Moreover, there exists $c>0$  such that
		\begin{align*}
			{\| y_{u}^{\tau,\eta}\|}_{W(\tau,T)} \le c\, \left({\|u-f(0)\|}_{L^2(Q_\tau)} + {\|\eta\|}_H\right)\quad\text{for all } u\in\Utad. 
		\end{align*}
	\end{proposition}
	
	\begin{proof}
		See \cite[Lemma 5.3 on p. 373]{Tro2010} and \cite[Lemma 7.10]{Tro2010}.
	\end{proof}
	
	Let us now point out that the objective functionals $\mathcal J_{\tau,\eta}$ are well defined as each state $y_{u}^{\tau,\eta}$ and the target state $y_Q$ both belong to $L^2(Q)$. One can achieve a further refinement of the regularity of states beyond mere square integrability.
	
	\begin{proposition} \label{Prop3.2}
		For any $\eta\in H$, $\tau\in[0,T]$ and $u\in\Utad$, the state $y_{u}^{\tau,\eta}$ belongs to $L^{\nicefrac{2(d+2)}{d}}(Q_\tau)$. 
	\end{proposition}
	
	\begin{proof}
		By \cite[Lemma 2.2]{CW_2023}, the space $L^2(\tau,T;V)\cap L^\infty(\tau,T; H)$ is continuously embedded in $L^{\nicefrac{2(d+2)}{d}}(Q_\tau)$, whence the result follows.
	\end{proof}
	  {The proof of the next proposition follows standard arguments. We give a short sketch of the arguments needed, since here we need to account for the situation where the initial condition is only in $L^2(\Omega)$.}
	\begin{proposition}\label{prop3.3}
		Each problem \cref{Pteta} admits a global minimizer. 
	\end{proposition}
	
	\begin{proof}
{Let $I=[\tau,T]$ and recall $W(I)$ defined before. It is well-known that $W(I)\hookrightarrow C(I;L^2(\Omega))$ continuously; hence $y(\tau)\in L^2(\Omega)$ is well-defined. In particular, we recall the weak formulation for the state with $L^2$ initial condition, \eqref{pointwiseweak1}.
By Assumption~2.1, the box bounds satisfy $u_a,u_b\in L^\infty(Q)$; therefore
 $U^{\tau}_{\mathrm{ad}}$ is a nonempty, convex, closed, and bounded subset of $L^2(Q_\tau)$. For each $u\in U^{\tau}_{\mathrm{ad}}$ and each $\eta\in H$, there exists a unique weak solution $y_{\tau,\eta}^u\in W(\tau,T)$ of \eqref{pointwiseweak1} and the estimate $\|y_{\tau,\eta}^u\|_{W(\tau,T)}\ \le\ c\big(\|u-f(0)\|_{L^2(Q_\tau)}+\|\eta\|_{L^2(\Omega)}\big)$ holds. 
As a consequence, $y_{\tau,\eta}^u\in L^2(Q_\tau)$, so $J_{\tau,\eta}(u)$ is well-defined.
Choose a minimizing sequence $(u_k)\subset U^{\tau}_{\mathrm{ad}}$ with
$J_{\tau,\eta}(u_k)\to \inf_{U^{\tau}_{\mathrm{ad}}} J_{\tau,\eta}$.
Thus, $(u_k)$ is bounded in $L^2(Q_\tau)$; extracting a subsequence, denoted in the same way, $u_k \rightharpoonup u$ weakly in  $L^2(Q_\tau)$ for some $u\in U^{\tau}_{\mathrm{ad}}$. Let $y_k:=y_{\tau,\eta}^{u_k}$ and $y:=y_{\tau,\eta}^{u}$. 
Then $(y_k)$ is bounded in $W(\tau,T)$ uniformly in $k$.  Up to a subsequence, by the Aubin-Lions theorem and the uniform bound in $W(\tau,T)$
\[
y_k \to \widehat y \ \text{ strongly in }L^2(Q_\tau),\qquad
y_k \rightharpoonup \widehat y \ \text{ weakly in }W(\tau,T).
\]
Since $f\in C^2$ with bounded derivative by \cref{Assumption}, $f$ is globally Lipschitz; hence $f(y_k)\to f(\widehat y)$ in $L^2(Q_\tau)$. Passing to the limit in the weak formulation with $u_k\rightharpoonup u$ shows that $\widehat y$ solves the state equation with control $u$ and initial datum $\eta$, hence $\widehat y=y$ by uniqueness. In particular, $y_k \to y \quad \text{in }L^2(Q_\tau)$. By the strong convergence $y_k\to y$ in $L^2(Q_\tau)$, $J_{\tau,\eta}(u)\ \le\liminf_{k\to\infty} J_{\tau,\eta}(u_k)$. Thus $u$ is a global minimizer of $J_{\tau,\eta}$ over $U^{\tau}_{\mathrm{ad}}$.}
	\end{proof}

	\subsection{A technical result on stability}
	%%%%%%%%%%%%%%%%
	
	Let us begin with a lemma showing that local growth at a minimizer can become global, provided the minimizer is unique.
	
	\begin{lemma}
		\label{globalgrowth}
		Suppose that $\bar u\in\Uad$ is a unique global minimizer of \eqref{P} that satisfies \cref{MA}. Then there exists $\gamma>0$ such that
		\begin{align}
			\label{ebne}
			\mathcal J(u) - 	\mathcal J(\bar u) \ge \gamma\,{\|y_u-\bar y\|}_{L^2(Q)}^{2}\quad\text{for all }u\in \Uad. 
		\end{align}
	\end{lemma}
	
	\begin{proof}
		By \cref{Proploc}, there exist numbers $\delta>0$ and $c>0$ such that
		\begin{align}
			\mathcal J(u) - \mathcal J(\bar u) \ge c\,{\| y_{u} - y_{\bar u}\|}_{L^2(Q)}^2
		\end{align}
		for all $u\in\Uad$ such that $\|y_{u}-y_{\bar u}\|_{L^2(Q)}\le\delta$.
		We now claim that there exists $\varepsilon>0$ such that $\mathcal J(u)-\mathcal J(\bar u)\ge \varepsilon$ for all $u\in\Uad$ such that $\|y_u-y_{\bar u}\|_{L^2(Q)}\ge \delta$. Otherwise, there would exist a sequence $\{u_n\}_{n\in\Nb}\subset\Uad$ satisfying $\|y_{u_n}-y_{\bar u}\|_{L^2(Q)}\ge \delta$ and $\mathcal J(u_n)\to \mathcal J(\bar u)$ as $n\to +\infty$. We can extract a subsequence $\{u_{n_k}\}_{k\in\mathbb N}$ of $\{u_n\}_{n\in\Nb}$ converging weakly in $L^2(Q)$ to some $\hat u\in\Uad$. We argue that $\{y_{u_{n_k}}\}_{k\in\mathbb N}$ converges weakly in $L^2(Q)$ to $y_{\hat u}$. {By Proposition \ref{Prop3.1} the associated states $\{y_{u_{n_k}}\}$ are uniformly bounded in $W(0,T)$, hence by Aubin–Lions theorem relatively compact in $L^2(Q)$. Thus, up to a subsequence, $y_{u_{n_k}}\to \widehat y$ in $L^2(Q)$ and weakly in $W(0,T)$. Since $f$ is globally Lipschitz by Assumption \ref{Assumption}, $f(y_{u_{n_k}})\to f(\widehat y)$ in $L^2(Q)$. Passing to the limit in the weak formulation shows that $\widehat y=y_{\hat u}$. Hence $y_{u_{n_k}}\to y_{\hat u}$ strongly in $L^2(Q)$.} Using {the} lower semicontinuity of the squared $L^2(Q)$-norm, 
		\begin{align*}
			\mathcal J(\hat u) \le \liminf_{k\to+\infty} \mathcal J(u_{n_k}) =\lim_{k\to+\infty} \mathcal J(u_{n_k}) = \mathcal J(\bar u).
		\end{align*}
		Since $\bar u$ is a unique minimizer, we deduce that $\hat u=\bar u$, and that consequently  $y_{u_n}\to y_{\bar u}$ in $L^2(Q)$, yielding a contradiction.  Set  $M\colonequals\sup_{u\in\Uad}\| y_{u} - \bar y \|_{L^2(Q)}$. Then, for all $u\in\Uad$ with $\|y_u-y_{\bar u}\|_{L^2(Q)}\ge\delta$,
		\begin{align*}
			\mathcal J(u)-\mathcal J(\bar u)\ge \frac{\varepsilon}{M^2}\,{\| y_{u} - y_{\bar u}\|}_{L^2(Q)}^2.
		\end{align*}
		Setting $\gamma\colonequals\min\{c,\nicefrac{\varepsilon}{M^2}\}$, we see that
		\begin{align*}
			\mathcal J(u)-\mathcal J(\bar u)\ge \gamma\,{\|y_u-y_{\bar u}\|}_{L^2(Q)}^2\quad\text{for all } u\in\Uad.
		\end{align*}
	\end{proof}
	
	\begin{corollary}
		\label{corlem}
		Suppose that $\bar u\in\Uad$ is a unique global minimizer  satisfying \cref{MA}. Then, there exists $\gamma>0$ such that, for any $\tau\in[0,T]$, 
		\begin{align}
			\mathcal J_{\tau, y_{\bar u}(\tau)}(u) - 	\mathcal J_{\tau, y_{\bar u}(\tau)}(\bar u) \ge \gamma\left\|y^{\tau,y_{\bar u}(\tau)}_{u}-y^{{\tau,y_{\bar u}(\tau)}}_{\bar u}\right\|_{L^2(Q_\tau)}^2\quad \text{for all } u\in\Utad.
		\end{align}
	\end{corollary}
	
	\begin{proof}
		Given $u\in\Utad$, define the control $w\in\Uad$ as  $w(t)\colonequals\bar u(t)$ if $t\in [0,\tau)$ and $w(t)\colonequals u(t)$ if $t \in [\tau,T]$. Then it holds $y_w=y_{\bar u}$ for $t\in [0,\tau]$ and $y_w=y_{u}^{\tau,y_{\bar u}(\tau)}$ for $t\in [\tau ,T]$. The result is obtained by substituting $u=w$ in \cref{ebne}. 
	\end{proof}
	
	We will recurrently use the following lemma. 
	
	\begin{lemma}
		\label{lemma:LsL1}
		Let $\tau\in[0,T)$. For given $\alpha\in L^\infty(Q_\tau)$ and $\rho\in L^2(Q_\tau)$ there exists a unique weak solution $z_{\alpha,\rho}\in W(\tau,T)$ to
		\begin{align}
			\label{testequ}
			\begin{cases}
				\displaystyle\frac{\partial z}{\partial t}- \dive\big(A(\bx)\nabla z\big) + \alpha z = \rho\text{ a.e. in }Q_\tau,\\
				z=0\text{ a.e. on } \Sigma_\tau,\quad  z(\tau) =0 \text{ a.e. in }\Omega.
			\end{cases}
		\end{align}
		Furthermore, for each $s\in[1,\nicefrac{(d+2)}{d})$ there exists $c>0$ such that, for any $\tau\in[0,T)$,
		\begin{align*}
			{\| z_{\alpha,\rho}\|}_{L^s(Q_\tau)}\le c {\|\rho\|}_{L^1(Q_\tau)}
		\end{align*}
		for all nonnegative $\alpha\in L^\infty(Q_\tau)$, and  all $\rho\in L^2(Q_\tau)$.
	\end{lemma}
	\begin{proof}
		Existence and uniqueness  of weak solutions  for (\ref{testequ}) 
		follow from standard results; see, e.g., \cite[Lemma 5.3 on p. 373]{Tro2010}.  By  \cite[Lemma 2.2]{CW_2023}, the space $L^2(\tau,T;V)\cap L^\infty(\tau,T; H)$ is continuously embedded in $L^{\nicefrac{2(d+2)}{d}}(Q_\tau)$. Consequently, for any $\rho\in L^2(Q_\tau)$ and $\alpha\in L^\infty(Q_\tau)$, $z_{\alpha,\rho}$ belongs to $L^{\nicefrac{2(d+2)}{d}}(Q_\tau)$, and in particular to $L^s(Q_\tau)$ for any $s\in[1,\nicefrac{(d+2)}{d})$.  From this, the estimate follows straightforwardly, see, e.g., the argument in the proof of  \cite[Lemma 2]{DJV_2023}. The independence of the constant $c > 0$ from $\tau$ can be easily verified by applying the result for the case $\tau = 0$ to the other cases.
	\end{proof}
	
	Recall that, under the notation convention we are using, there holds
	\begin{align*}
		 \text{$y_{\bar u} = y_{\bar u}^{\tau, y_{\bar u}(\tau)}$\,\, in $Q_{\tau}$\,\, for any $\tau\in[0,T)$. }
	\end{align*}
	Recall that $y_{\bar u}$ denotes the reference state. Next, we present a couple of technical lemmas that involve non-standard estimates related to stability.
	\begin{lemma}
		\label{eqesls}
		Let $\tau\in(0,T)$. Define $z_{u,h,\eta}\in W(\tau,T)$ as
		\begin{align*}
			z_{u,h,\eta}\colonequals y_{u}^{\tau+h,y_{\bar u}(\tau+h)}  - y_{u}^{\tau+h,y_{\bar u}(\tau)+\eta} + y_{\bar u}^{\tau+h,y_{\bar u}(\tau)+\eta} - y_{\bar u}^{\tau+h,y_{\bar u}(\tau+h)}
		\end{align*}
		for any $\eta\in H$, $h\in(0,T-\tau)$ and  $u\in\Uthad$. Then, for each $s\in[1,\nicefrac{d+2}{d})$ there exists a constant $c>0$ (independent of $u,h$ and $\eta$)  such that 
		\begin{align*}
			&{\|z_{u,h,\eta}\|}_{L^s(Q_{\tau+h})} \le c\Big({\|y_{\bar u}(\tau+h)-y_{\bar u}(\tau)\|}_{L^2(Q)}^2 +{\|\eta\|}^2_H\Big)\\
			&\hspace{1.4cm}+c\,\Big({\|y_{\bar u}(\tau+h)-y_{\bar u}(\tau)\|}_H+{\|\eta\|}_H \Big){\|y^{\tau+h,y_{\bar u}(\tau+h)}_{u}-y_{\bar u}^{\tau+h,y_{\bar u}(\tau+h)}\|}_{L^2(Q_{\tau+h})}.
		\end{align*}
	\end{lemma} 	
	
	\begin{proof}
		Let $h\in(0,T-\tau)$, $\eta\in H$ and $u\in\Uthad$. Applying the Mean Value Theorem, we can find functions  $\theta_{u,h,\eta},\theta_{\bar u,h,\eta}:Q_{\tau+h}\to[0,1]$ such that
		\begin{align*}
			f\big(y_{u}^{\tau+h,y_{\bar u}(\tau+h)}\big)- f\big(y_{u}^{\tau+h,y_{\bar u}(\tau)+\eta}\big) &=f'\big({y_{\theta_{u,h,\eta}}}\big)\big(y_{u}^{\tau+h,y_{\bar u}(\tau+h)}-y_{u}^{\tau+h,y_{\bar u}(\tau)+\eta}\big),\\
			f\big(y_{\bar u}^{\tau+h,y_{\bar u}(\tau+h)}\big)- f\big(y_{\bar u}^{\tau+h,y_{\bar u}(\tau)+\eta}\big) &=f'\big({y_{\theta_{\bar u,h,\eta}}}\big)\big(y_{\bar u}^{\tau+h,y_{\bar u}(\tau+h)}-y_{\bar u}^{\tau+h,y_{\bar u}(\tau)+\eta}\big)
		\end{align*}
		on $Q_{\tau+h}$ with 
		\begin{align*}
			y_{\theta_{u,h,\eta}}&\colonequals y_{u}^{\tau+h,y_{\bar u}(\tau+h)}+\theta_{u,h,\eta}\big(y_{u}^{\tau+h,y_{\bar u}(\tau)+\eta}-y_{u}^{\tau+h,y_{\bar u}(\tau+h)}\big),\\
			y_{ \theta_{\bar u,h,\eta}}&\colonequals y_{\bar u}^{\tau+h,y_{\bar u}(\tau+h)}+\theta_{\bar u,h,\eta}\big(y_{\bar u}^{\tau+h,y_{\bar u}(\tau)+\eta}-y_{\bar u}^{\tau+h,y_{\bar u}(\tau+h)}\big).
		\end{align*}
		Observe that 
		\begin{align*}
			&f\big(y_{u}^{\tau+h,y_{\bar u}(\tau+h)}\big)  - f\big(y_{u}^{\tau+h,y_{\bar u}(\tau)+\eta}\big)+ f\big(y_{\bar u}^{\tau+h,y_{\bar u}(\tau)+\eta}\big) - f\big(y_{\bar u}^{\tau+h,y_{\bar u}(\tau+h)}\big)\\
			&=f'\big({y_{\theta_{u,h,\eta}}}\big)\big(y_{u}^{\tau+h,y_{\bar u}(\tau+h)}-y_{u}^{\tau+h,y_{\bar u}(\tau)+\eta}\big)\\
			&\quad- f'\big({y_{\theta_{\bar u,h,\eta}}}\big)\big(y_{\bar u}^{\tau+h,y_{\bar u}(\tau+h)}-y_{\bar u}^{\tau+h,y_{\bar u}(\tau)+\eta}\big)\\
			&=f'\big({y_{\theta_{u,h,\eta}}}\big)z_{u,h,\eta} +\left( f'\big({y_{\theta_{u,h,\eta}}}) -  f'\big({y_{\theta_{\bar u,h,\eta}}}\big)\right)\big(y_{\bar u}^{\tau+h,y_{\bar u}(\tau+h)}-y_{\bar u}^{\tau+h,y_{\bar u}(\tau)+\eta}\big).
		\end{align*}
		Defining
		\begin{align*}
			\rho_{u,h,\eta}\colonequals\left( f'\big({y_{\theta_{\bar u,h,\eta}}}\big) -  f'\big({y_{\theta_{ u,h,\eta}}}\big)\right)\big(y_{\bar u}^{\tau+h,y_{\bar u}(\tau+h)}-y_{\bar u}^{\tau+h,y_{\bar u}(\tau)+\eta}\big)
		\end{align*}
		we see that
		\begin{align}\label{eq:sumofpde1}
			\begin{cases}
				\displaystyle\frac{\partial z_{u,h,\eta}}{\partial t}- \dive\big(A(\bx)\nabla z_{u,h,\eta}\big) + f'(y_{\theta_{u,h,\eta}} ) z_{u,h,\eta}= \rho_{u,h,\eta}\text{ a.e. in } Q_{\tau+h}, \\
				z_{u,h,\eta}=0\text{ a.e. on } \Sigma_{\tau+h},\quad  z_{u,h,\eta}(\tau+h)=0\text{ a.e. in } \Omega. 
			\end{cases}
		\end{align}
		Since by \cref{Assumption}-$(ii)$  $f':\mathbb R\to\mathbb R$ is bounded, we see that $\alpha:=f'(y_{\theta_{u,h,\eta}})$ belongs to $L^\infty(Q_{\tau+h})$ and $\rho$ to $L^2(Q_{\tau+h})$. 
        {From the boundedness of $f''$, which holds according to \cref{Assumption}, we infer the existence of a constant $c>0$ independent of the specific $\hat y$, such that $\|f''(\hat y)\|_{L^\infty(Q_{\tau+h})}\leq c$, for all $\hat y$ obtained from Taylor's theorem, such that $f'\big({y_{\theta_{\bar u,h,\eta}}}\big) -  f'\big({y_{\theta_{ u,h,\eta}}}\big)=f''(\hat y)\left({y_{\theta_{\bar u,h,\eta}}}-{y_{\theta_{ u,h,\eta}}}\right)$. We estimate ${\|\rho_{u,h,\eta}\|}_{L^1(Q_{\tau+h})}\leq c \|{y_{\theta_{\bar u,h,\eta}}}-{y_{\theta_{ u,h,\eta}}} \|_{L^2(Q_{\tau+h})}\|y_{\bar u}^{\tau+h,y_{\bar u}(\tau+h)}-y_{\bar u}^{\tau+h,y_{\bar u}(\tau)+\eta} \|_{L^2(Q_{\tau+h})}$. Using the definitions of the intermediate states, we write the corresponding differences of states, by means of the mean-value theorem, as solutions of linear parabolic equations with bounded coefficients. Applying the associated a priori estimates, we obtain that the first term on the right-hand side can be estimated by
        \begin{align*}
           & \|{y_{\theta_{\bar u,h,\eta}}}-{y_{\theta_{ u,h,\eta}}} \|_{L^2(Q_{\tau+h})}\\
           &\leq c \|y_{u}^{\tau+h,y_{\bar u}(\tau+h)}-y_{\bar u}^{\tau+h,y_{\bar u}(\tau+h)}\|_{L^2(Q_{\tau+h})}+2c \|y_{\bar u}(\tau)+\eta-y_{\bar u}(\tau+h) \|_{L^2(Q_{\tau+h})}.
        \end{align*}
        Again, utilizing a priori estimates for linear equations, the second term on the right-hand side estimates by
        \begin{align*}
            \|y_{\bar u}^{\tau+h,y_{\bar u}(\tau+h)}-y_{\bar u}^{\tau+h,y_{\bar u}(\tau)+\eta} \|_{L^2(Q_{\tau+h})}\leq c (\|y_{\bar u}(\tau)-y_{\bar u}(\tau+h) \|_{L^2(Q_{\tau+h})}+\| \eta \|_H).
        \end{align*}
		Now, combining these estimates and applying the Young inequality, it is straightforward to obtain} 
		\begin{align*}
			&{\|\rho_{u,h,\eta}\|}_{L^1(Q_{\tau+h})}\le  c\Big({\|y_{\bar u}(\tau+h)-y_{\bar u}(\tau)\|}_{L^2(Q)}^2 +{\|\eta\|}^2_H\Big)\\
			&\hspace{1.3cm}+c\,\Big({\|y_{\bar u}(\tau+h)-y_{\bar u}(\tau)\|}_H+{\|\eta\|}_H \Big){\|y^{\tau+h,y_{\bar u}(\tau+h)}_{u}-y_{\bar u}^{\tau+h,y_{\bar u}(\tau+h)}\|}_{L^2(Q_{\tau+h})}.
		\end{align*}
		for a constant $c>0$ which is independent of $u,h$ and $\eta$. The result then follows directly from Lemma~\ref{lemma:LsL1}. 
		\end{proof}
		\iffalse 
		\begin{align}\label{genin1}
			\|z_{u,h}\|_{L^s(Q_{\tau+h})}&\le c'\|y_{\bar u}(\tau+h)-y_{\bar u}(\tau)\|_{L^2(\Omega)} \|y_{u}-y_{\bar u}\|_{L^2(Q_{\tau+h})}\\
			&\hspace*{1.4cm}+c' \|y_{\bar u}(\tau+h)-y_{\bar u}(\tau)\|_{L^2(\Omega)}^2\nonumber
		\end{align}
		%
		for a constant $c'>0$ independent of $u$ and $h$.  Denote now
		%
		\begin{align*}
			z_{u,\eta}=y_{\bar u}^{\tau+h,y_{\bar u}(\tau)} - y_{\bar u}^{\tau+h,y_{\bar u}(\tau+h)} + y_{\bar u}^{\tau+h,y_{\bar u}(\tau)+\eta} - y_{\bar u}^{\tau+h,y_{\bar u}(\tau+h)}.
		\end{align*}
		%
		Repeating the same arguments, one can estimate 
		\begin{align}\label{genin2}
			\|z_{u,\eta}\|_{L^s(Q_{\tau+h})}&\le c''\Big(\|\eta\|_{L^2(\Omega)} \|y_{u}-y_{\bar u}\|_{L^2(Q_{\tau+h})}+ \|\eta\|_{L^2(\Omega)}^2\Big)
		\end{align}
		%
		for$c''>0$ independent of $u$ and $h$. Using the triangle inequality and estimates \cref{genin1} and \cref{genin2}  yield the result. 
		\fi

	\begin{lemma}\label{esobfu}
		Let $\tau\in(0,T)$. Define $A_{u,h,\eta}\in \mathbb R$ as
		\begin{align*}
			A_{u,h,\eta}\colonequals\mathcal J_{\tau+h,y_{\bar u}(\tau+h)}(u)-\mathcal J_{\tau+h,y_{\bar u}(\tau)+\eta}(u)+\mathcal J_{\tau+h,y_{\bar u}(\tau)+\eta}(\bar u)-\mathcal J_{\tau+h,y_{\bar u}(\tau+h)}(\bar u).
		\end{align*}
		There exists a constant $c>0$ such that 
		\begin{align*}
			|A_{u,h,\eta}|&\le c\, \big({\|y_{\bar u}(\tau+h)-y_{\bar u}(\tau)\|}_H+{\|\eta\|}_H\big){\|y^{\tau+h,y_{\bar u}(\tau+h)}_{u}-y_{\bar u}^{\tau+h,y_{\bar u}(\tau+h)}\|}_{L^2(Q_{\tau+h})}\\
			&\quad+ c\,\big({\|y_{\bar u}(\tau+h)-y_{\bar u}(\tau)\|}_H^2 +{\|\eta\|}^2_H\big)
		\end{align*}
		for all $\eta\in H$ small enough, $h\in(0,T-\tau]$ and $u\in\Uthad$.
	\end{lemma} 	
	
	\begin{proof}
		Let $h\in(0,T-\tau)$, $\eta\in H$ with $\|\eta\|_H\le 1$ and $u\in\Uthad$. By \cref{Assumption}-(iii), there exists $r'>1+\nicefrac{d}{2}$ such that $y_Q\in L^{r'}(Q)$. We set $r\colonequals\min\{r',\nicefrac{2(d+2)}{d}\}$. Let $s>0$ so that $\nicefrac{1}{r}+\nicefrac{1}{s}=1$ and observe that $s<\nicefrac{d+2}{d}$. By $z_{u,h,\eta}\in L^s(Q)$ we denote the function in Lemma~\ref{eqesls}. Observe that
		\begin{align*}
			&A_{u,h,\eta}\\
			&~=\int_Q\big(y_{u}^{\tau+h,y_{\bar u}(\tau+h)}+y_{u}^{\tau+h,y_{\bar u}(\tau)+\eta}-2y_{Q}\big)\big(y_{u}^{\tau+h,y_{\bar u}(\tau+h)}-y_{u}^{\tau+h,y_{\bar u}(\tau)+\eta}\big) \,\dx\dt\\
			&~\quad-\int_Q\big(y_{\bar u}^{\tau+h,y_{\bar u}(\tau+h)}+y_{\bar u}^{\tau+h,y_{\bar u}(\tau)+\eta}-2y_{Q}\big)\big(y_{\bar u}^{\tau+h,y_{\bar u}(\tau+h)}-y_{\bar u}^{\tau+h,y_{\bar u}(\tau)+\eta}\big)\,\dx \dt\\
			&~=\int_Q\big(y_{u}^{\tau+h,y_{\bar u}(\tau+h)}+y_{u}^{\tau+h,y_{\bar u}(\tau)+\eta}-2y_{Q}\big)z_{u,h,\eta}\,\dx \dt\\
			&~\quad+\int_Q\big(y_{\bar u}^{\tau+h,y_{\bar u}(\tau+h)}-y_{\bar u}^{\tau+h,y_{\bar u}(\tau)+\eta}\big)w_{u,h,\eta}\,\dx\dt\\
			&~\equalscolon B_{u,h,\eta} + C_{u,h,\eta},
		\end{align*}
		where
		\begin{align*}
			w_{u,h,\eta}\colonequals y_{u}^{\tau+h,y_{\bar u}(\tau+h)}+y_{u}^{\tau+h,y_{\bar u}(\tau)+\eta}-y_{\bar u}^{\tau+h,y_{\bar u}(\tau+h)}-y_{\bar u}^{\tau+h,y_{\bar u}(\tau)+\eta}\in W(\tau_h,T).
		\end{align*}
		Using the continuous embeddings $W(\tau+h,T)\hookrightarrow L^{2(d+2)/d}(Q_{\tau+h})\hookrightarrow L^r(Q_{\tau+h})$, we can find a constant $M>0$ independent of $u,h$ and $\eta$ such that 
		\begin{align*}
			{\|y_{u}^{\tau+h,y_{\bar u}(\tau+h)}+y_{u}^{\tau+h,y_{\bar u}(\tau)+\eta}-2y_{Q}\|}_{L^{r}(Q_{\tau+h})} \le M.
		\end{align*}
		Employing  Lemma~\ref{eqesls}, we can conclude that there exists a constant $c_1>0$  independent of $u,h,\eta$ such that 
		\begin{align*}
			B_{u,h,\eta}&\le c_1 \Big({\|y_{\bar u}(\tau+h)-y_{\bar u}(\tau)\|}_H+{\|\eta\|}_H\Big){\|y^{\tau+h,y_{\bar u}(\tau+h)}_{u}-y_{\bar u}^{\tau+h,y_{\bar u}(\tau+h)}\|}_{L^2(Q_{\tau+h})}\\
			&\quad+ c_1\Big({\|y_{\bar u}(\tau+h)-y_{\bar u}(\tau)\|}_H^2 +{\|\eta\|}^2_H\Big).
		\end{align*}
		In a similar fashion, we can derive an analogous estimate for the term $C_{u,h,\eta}$. 
	\end{proof}

	We continue with a result that is a key component for the proof of our main theorem.
	
	\begin{theorem}
		\label{stability}
		Suppose that $\bar u\in\Uad$ is a unique global minimizer satisfying \cref{MA}. For each $\tau\in[0,T]$ there exists a constant $\kappa>0$ such that
		\begin{align*}
			{\| y^{\tau+h,y_{\bar u}(\tau)+\eta}_{u_{h,\eta}}- y^{\tau+h,y_{\bar u}(\tau+h)}_{\bar u}\|}_{L^2(Q_{\tau+h})}\le \kappa \Big({\|y_{\bar u}(\tau+h)-y_{\bar u}(\tau)\|}_H+{\|\eta\|}_H\Big)
		\end{align*}
		for all $h\in[0,T-\tau]$, $\eta\in H$ small enough, and any minimizer {$u_{h,\eta}\in\Uthad$} of problem  $(\mathbf P_{\tau+h,y_{\bar u}(\tau)+\eta})$.
	\end{theorem}
	
	\begin{proof}
		Let $h\in[0,T-\tau]$ (without loss of generality we assume $h\in(0,T-\tau)$), $\eta\in H$ small enough so that Lemma~\ref{esobfu} holds. Let  {$u=u_{h,\eta}\in\Uthad$} be a minimizer of problem $(\mathbf P_{\tau+h,y_{\bar u}(\tau)+\eta})$. Due to Corollary \ref{corlem} there exists a constant $\gamma>0$ (independent of $u,h,\eta$) such that 
		\begin{equation}
			\label{lebinp}
			\gamma\,\big\| y^{\tau+h,y_{\bar u}(\tau+h)}_u - y^{\tau+h,y_{\bar u}(\tau+h)}_{\bar u}\big\|_{L^2(Q_{\tau+h})}^2 \le \mathcal J_{\tau+h,y_{\bar u}(\tau+h)}(u)-\mathcal J_{\tau+h,y_{\bar u}(\tau+h)}(\bar u).
		\end{equation}
		Since $u$ is a minimizer of problem $(\mathbf P_{\tau+h,y_{\bar u}(\tau)+\eta})$, we have
		\begin{align}
			\label{ansch}
			0\le\mathcal J_{\tau+h,y_{\bar u}(\tau)+\eta}(\bar u)-\mathcal J_{\tau+h,y_{\bar u}(\tau)+\eta}(u).
		\end{align}
		Adding inequalities \eqref{lebinp} and \eqref{ansch}, and employing Lemma~\ref{esobfu}, we obtain
		\begin{align*}
			&\gamma\,{\| y^{\tau+h,y_{\bar u}(\tau+h)}_u - y^{\tau+h,y_{\bar u}(\tau+h)}_{\bar u}\|}_{L^2(Q_{\tau+h})}^2\\
			&~\leq \mathcal J_{\tau+h,y_{\bar u}(\tau+h)}(u)-\mathcal J_{\tau+h,y_{\bar u}(\tau)+\eta}(u)+\mathcal J_{\tau+h,y_{\bar u}(\tau)+\eta}(\bar u)-\mathcal J_{\tau+h,y_{\bar u}(\tau+h)}(\bar u)\\
			&~\le c \left({\|y_{\bar u}(\tau+h)-y_{\bar u}(\tau)\|}_H+{\|\eta\|}_H \right){\|y^{\tau+h,y_{\bar u}(\tau+h)}_{u}-y_{\bar u}^{\tau+h,y_{\bar u}(\tau+h)}\|}_{L^2(Q_{\tau+h})}\\
			&\quad~+ c\left({\|y_{\bar u}(\tau+h)-y_{\bar u}(\tau)\|}_H^2+{\|\eta\|}^2_H\right)
		\end{align*}
		for a constant $c>0$ which does not depend on $u,h$ and $\eta$. {Applying Young’s inequality, $ab \leq \nicefrac{a^2}{2\varepsilon}+\nicefrac{\varepsilon b^2}{2}$ for $\varepsilon>0$, for $a={\|y_{\bar u}(\tau+h)-y_{\bar u}(\tau)\|}_H+{\|\eta\|}_H$ and $b=\|y^{\tau+h,y_{\bar u}(\tau+h)}_{u}-y_{\bar u}^{\tau+h,y_{\bar u}(\tau+h)}\|_{L^2(Q_{\tau+h})}$ with $\varepsilon=\nicefrac{\gamma}{c}$ to the first term on the right-hand side  yields  
        \begin{align*}
           & \frac{\gamma}{2}\,{\| y^{\tau+h,y_{\bar u}(\tau+h)}_u - y^{\tau+h,y_{\bar u}(\tau+h)}_{\bar u}\|}_{L^2(Q_{\tau+h})}^2\\
           &\leq \frac{c(c+\gamma)}{\gamma}\left({\|y_{\bar u}(\tau+h)-y_{\bar u}(\tau)\|}_H^2+{\|\eta\|}^2_H\right).
        \end{align*}
        Rearranging the constants and taking the root gives for $\kappa:= \sqrt{\frac{2c(c+\gamma)}{\gamma^2}}$,
		\begin{align*}
			\| y^{\tau+h,y_{\bar u}(\tau+h)}_{u_{}}- y^{\tau+h,y_{\bar u}(\tau+h)}_{\bar u}\|_{L^2(Q_{\tau+h})}\le \kappa \Big(\|y_{\bar u}(\tau+h)-y_{\bar u}(\tau)\|_{H} + \|\eta\|_{H} \Big).
		\end{align*}
        }
        {The result then follows from the above, since
        \begin{align*}
            &\| y^{\tau+h,y_{\bar u}(\tau)+\eta}_{u_{h,\eta}}- y^{\tau+h,y_{\bar u}(\tau+h)}_{\bar u}\|_{L^2(Q_{\tau+h})}\leq \| y^{\tau+h,y_{\bar u}(\tau)+\eta}_{u_{h,\eta}}- y^{\tau+h,y_{\bar u}(\tau+h)}_{u_{h,\eta}}\|_{L^2(Q_{\tau+h})}\\
            &+{\| y^{\tau+h,y_{\bar u}(\tau+h)}_{u_{h,\eta}}- y^{\tau+h,y_{\bar u}(\tau+h)}_{\bar u}\|}_{L^2(Q_{\tau+h})}\leq \kappa' \Big(\|y_{\bar u}(\tau+h)-y_{\bar u}(\tau)\|_{H} + \|\eta\|_{H} \Big),
        \end{align*}
        for $\kappa'=(c_2+\kappa)$, where we used the a priori estimate
		\begin{align*}
			 \| y^{\tau+h,y_{\bar u}(\tau)+\eta}_{u_{h,\eta}}- y^{\tau+h,y_{\bar u}(\tau+h)}_{u_{h,\eta}}\|_{L^2(Q_{\tau+h})}\leq c_2  \Big(\|y_{\bar u}(\tau+h)-y_{\bar u}(\tau)\|_{H} + \|\eta\|_{H} \Big)
		\end{align*}
		which holds for some constant $c_2>0$ independent of $u,h,\eta$.}
	\end{proof}
        {
         The following identity, while classical in optimal control, underlies the differentiability analysis in Section \ref{proofss}, for which reason, we give a short comment on its derivation. For a given control $\bar u$ with associated state $\bar y$, the adjoint $\bar p$, defined as the weak solution on $[\tau,T]$ of
        \begin{align*}
			\begin{cases}
				\displaystyle-\frac{\partial p}{\partial t}- \dive\big(A(\bx)\nabla p\big) + f'(\bar y) p= \bar y - y_Q\text{ a.e. in } Q_{\tau}, \\
				p=0\text{ a.e. on } \Sigma_{\tau},\quad  p(T)=0\text{ a.e. in } \Omega, 
			\end{cases}
		\end{align*}
    is by standard parabolic theory an element of $ W(\tau,T)$. Let $\bar z$ denote the linearized state associated with a perturbation $(\delta u,\zeta)$, given as the solution of
        \begin{align*}
			\begin{cases}
				\displaystyle \frac{\partial z}{\partial t}- \dive\big(A(\bx)\nabla z\big) + f'(\bar y) z= \delta u\text{ a.e. in } Q_{\tau}, \\
				z=0\text{ a.e. on } \Sigma_{\tau},\quad  z(\tau )=\zeta\text{ a.e. in } \Omega. 
			\end{cases}
		\end{align*}
        Integration-by-parts for elements of $W(\tau,T)$ then yields the very useful relation
        \[
        \int_{Q_\tau} (\bar y-y_Q) \bar z\,dx\,dt = \int_{Q_\tau} \bar p\,\delta u\,dx\,dt + \langle \bar p(\tau),\zeta\rangle_{L^2(\Omega)}.
        \]
        }
	%%%%%%%%%%%%%%%%
	\section{Proofs of \cref{thm3,thm2}}
    \label{proofss}
	%%%%%%%%%%%%%%%%
	
	%%%%%%%%%%%%%%%%

	%%%%%%%%%%%%%%%%%%%%%%%%%%%%%%%%%%%%%%%
	%\subsection{Proof of \cref{thm3}}
    
    \begin{proof}[Proof of \cref{thm3}]
	Without any loss of generality, assume that $\tau=0$. To simplify notation, we abbreviate $\upsilon(\eta)\colonequals\upsilon(0,\eta)$ for $\eta\in H$ and $L(t,\bx,y)\colonequals \nicefrac{1}{2}|y-y_Q(t,\bx)|^2$ for $(t,\bx)\in Q$. Recall that $\langle \cdot, \cdot\rangle$ denotes the duality paring/inner product in $H=L^2(\Omega)$.
	\smallbreak 
	We shall prove that 
	\begin{align*}
		&\limsup_{\eta\to 0} \frac{\upsilon(y_0+\eta)-\upsilon(y_0)-\langle p_{\bar u}(0),\eta\rangle}{\|\eta\|_H}\le 0\le \liminf_{\eta\to 0}\frac{\upsilon(y_0+\eta)-\upsilon(y_0)-\langle p_{\bar u}(0),\eta\rangle}{\|\eta\|_H}.
	\end{align*}
	The calculations will be carried in two parts, items \textbf{(I)} and \textbf{(II)} below. 
	\smallbreak 
	For each $\eta\in H$, let $u_\eta$ be a minimizer of $(\mathbf P_{0,y_0+\eta})$ and $y_\eta$ its associated state. Using \cref{stability} and the triangle inequality, one can estimate 
	\begin{align}\label{stakappa}
		\| y_\eta - y_{\bar u}\|_{L^2(Q)}\le \kappa \|\eta\|_{H} 
	\end{align}
	for some constant $\kappa>0$ independent of $\eta$. 
	\smallbreak 
	\textbf{(I)}  Let us begin observing that
	\begin{align}\label{Ep1}
		&\upsilon(y_0+\eta)-\upsilon(y_0)= \int_{{Q}} L(t,\bx,y_\eta(t,\bx))-L(t,\bx,y_{\bar u}(t,\bx))\,\dx\dt\\
		&= \int_{{Q}} L(t,\bx,y_\eta(t,\bx)) - L(t,\bx,y_{\bar u}(t,\bx)) - L_y(t,\bx,y_{\bar u}(t,\bx))\big(y_\eta(t,\bx) - y_{\bar u}(t,\bx)\big)\,\dx\dt\nonumber\\
		&\quad+\int_{{Q}} L_y(t,\bx,y_{\bar u}(t,\bx))\big(y_\eta(t,\bx) - y_{\bar u}(t,\bx)\big)\,\dx\dt\equalscolon A + B.\nonumber 
	\end{align}
	One can easily estimate, 
	 {
	\begin{align*}
		A&= \int_Q \frac{1}{2}\left( y_\eta (t,x)^2-y_{\bar u}(t,x)^2\right) - y_{\bar u}(t,x) y_\eta(t,x)\dx\dt \ge 0. 
	\end{align*}
    }
	Now, we can break term $B$ in two parts, 
	\begin{align*}
		B&=\int_{{Q}} L_y(t,\bx,y_{\bar u}(t,\bx))\big(y_\eta(t,\bx) - y_{\bar u}(t,\bx)\big)\,\dx\dt\\
		& = \int_{{Q}} L_y(t,\bx,y_{\bar u}(t,\bx))\big(y_\eta(t,\bx) - y_{\bar u}(t,\bx) - z_{\eta}(t,\bx)\big)\,\dx\dt\\
		&\hspace*{1.5cm} +  \int_{{Q}} L_y(t,\bx,y_{\bar u}(t,\bx)) z_{\eta}(t,\bx)\,\dx\dt\\
		&=: B_1 + B_2,
	\end{align*}
	where $z_\eta\in W(0,T)$ is the solution of equation 
	\begin{align}\label{eq:sumofpde111}
		\begin{cases}
			\displaystyle\frac{\partial z_{\eta}}{\partial t}- \dive\big(A(\bx)\nabla z_{\eta}\big) + f'(y_{\bar u} ) z_{\eta}= u_\eta - \bar u\text{ a.e. in } Q_{}, \\
			z_{\eta}=0\text{ a.e. on } \Sigma_{},\quad  z_{\eta}(0)=\eta\text{ a.e. in } \Omega. 
		\end{cases}
	\end{align}
Estimates of the type $L^s$-$L^1$ (see Lemma \ref{lemma:LsL1}) easily yield that
\begin{align*}
	\|y_\eta - y_{\bar u} - z_{\eta}\|_{L^s(Q)}\le c\| y_{\eta} - y_{\bar u} \|_{L^2(Q)}^2
\end{align*}
for some constant $c>0$ independent of $\eta$, where $s>1$  is such that $s^{-1}+r^{-1}=1$ with  $r:=\nicefrac{2(d+2)}{d}$. Combining this the stability estimate (\ref{stakappa}), we get 
\begin{align*}
	|B_1| \le c'\| y_{\eta} - y_{\bar u} \|_{L^2(Q)}^2 \le c'\kappa^2 \|\eta\|_{H}^2 
\end{align*}
for some constant $c'>0$ independent of $\eta$. Using  the integration parts formula and first-order necessary condition for $\bar u$, we get 
	\begin{align*}
	B_2&= \int_Q p_{\bar u}(t,\bx)(u_\eta(t,\bx) -\bar u(t,\bx))\,\dx\dt + {\langle p_{\bar u}(0),\eta \rangle}\ge {\langle p_{\bar u}(0),\eta \rangle}. 
\end{align*}
The two previous estimates allow us to conclude that $B\ge \langle p_{\bar u}(0),\eta\rangle + o\big(\|\eta\|_H\big)$. 
	Therefore, combining estimates $A$ and $B$,
	\begin{align*}
	\liminf_{\eta\to 0}	\frac{\upsilon(y_0+\eta)-\upsilon(y_0) - {\langle p_{\bar u}(0),\eta \rangle}}{\|\eta\|_H}& =
   {\liminf_{\eta\to 0}\frac{A+B_1}{\|\eta\|_H}+\liminf_{\eta\to 0}\frac{\langle p_{\bar u},u_\eta-\bar u\rangle}{\|\eta\|_H}} \geq 0.
	\end{align*}
This concludes with the subdifferentiability estimate claimed in this item.  
	\smallbreak 
   {The estimate for the opposite inequality is obtained by interchanging $(u_\eta,y_\eta)$ and $(\bar u,\bar y)$ in \eqref{Ep1}, using integration by parts with the adjoint $p_\eta$, and repeating the same algebraic decomposition and using fact that $\|p_{\eta}(0)-p_{\bar u}(0)\|_H\le c''\|\eta\|_H$ for some constant $c''>0$ independent of $\eta$.}
   \end{proof}

	\begin{proof}[Proof of \cref{thm2}]
	Let $\tau > 0$ be given. We first prove item (i). Item (ii) is then proved separately, using \cref{thm3}. For each $h > 0$, consider a minimizer $u_h \in \Uthad$ for problem $(\mathbf{P}_{\tau+h, y_{\bar{u}}(\tau)})$. Denote by $y_h$ and $p_h$ the state and costate associated with $u_h$. Note that $y_h(\tau+h) = y_{\bar{u}}(\tau)$. 
	\smallbreak 
	The proof is divided into three steps. First, we show $\|y_h - y_{\bar{u}}\|_{L^2(Q_{\tau+h})}^2 = o(h)$. In the subsequent steps, we use this to get sub- and super-differentiability estimates. %the upper and lower right Dini derivatives, see, e.g., \cite[Chapter 3]{KK_1996}. 
	To simplify notation, we abbreviate $\upsilon(t)\colonequals\upsilon(t,y_{\bar u}(\tau))$ for $t\in[0,T]$ and $L(t,\bx,y)\colonequals \nicefrac{1}{2}\,|y-y_Q(t,\bx)|^2$ for $(t,\bx)\in Q$. To simplify some calculations, and show the essence of the proof, we assume that $f(y)=y$; the arguments to consider a general nonlinearity satisfying \cref{Assumption}-$(ii)$  are standard, and outlined in the proof of \cref{thm3}.
	Step 1. By \cref{stability}, one gets $
		\|y_{u_h}^{\tau+h,y_{\bar u}(\tau+h)} - y_{\bar u}\|_{L^2(Q_{\tau+h})}^2\le \kappa\|y_{\bar u}(\tau+h)-y_{\bar u}(\tau)\|_H^2$ for some constant $\kappa>0$.
	Since $y_h$ and $y_{u_h}^{\tau+h,y_{\bar u}(\tau+h)}$ satisfy the same dissipative PDE, save for the initial condition at $\tau+h$,  we can easily obtain that there exists a constant $c>0$ such that
	\begin{align*}
		{\| y_{h}-y_{u_h}^{\tau+h,y_{\bar u}(\tau+h)}\|}_{L^2(Q_{\tau+h})}\le c\,{\|  y_{\bar u}(\tau+h) - y_{\bar u}(\tau) \|}_H.
	\end{align*}
	Combining the two previous estimates and applying the triangle inequality, we deduce
	\begin{align*}
		&{\|y_h - y_{\bar u}\|}_{L^2(Q_{\tau+h})}^2\le 2\,\| y_{h}-y_{u_h}^{\tau+h,y_{\bar u}(\tau+h)}\|_{L^2(Q_{\tau+h})}^2 + 2\,\|y_{u_h}^{\tau+h,y_{\bar u}(\tau+h)} - y_{\bar u}\|_{L^2(Q_{\tau+h})}^2\\
		&\leq 2(c^2+\kappa^2)\,\|  y_{\bar u}(\tau+h) - y_{\bar u}(\tau) \|_H^2 = 2(c^2+\kappa^2)\,\Big\|\frac{y_{\bar u}(\tau+h) - y_{\bar u}(\tau) }{h} \Big\|_H^2 \,h^2 
	\end{align*}
	Since $y_{\bar u}:[0,T]\to H$ is assumed to be right differentiable at $\tau$, 
	\begin{align*}
		h\hspace{0.04cm}\Big\|\frac{y_{\bar u}(\tau+h) - y_{\bar u}(\tau) }{h} \Big\|_H \longrightarrow 0\quad \text{as}\quad h\longrightarrow 0^+.
	\end{align*}
	Thus, we conclude that $\|y_h - y_{\bar u}\|_{L^2(Q_{\tau+h})}^2\le o(h)$.
	\smallbreak 
	{Step 2.} Next we prove that 
	\begin{align*}
		\liminf_{h\to 0^+} \frac{\upsilon(\tau+h)-\upsilon(\tau)}{h}\ge - \int_{\Omega} p_{\bar u}(\tau,\bx)\,\frac{\rmd^+ y_{\bar u}}{\dt}(\tau,\bx)\,\dx -   \int_\Omega L\big(\tau,\bx,y_{\bar u}(\tau,\bx)\big)\,\dx.
	\end{align*}
	First, we observe that
    \small
	\begin{align*}
		&\upsilon(\tau+h) - \upsilon(\tau)\\
		&= \int_{\tau+h}^{T}\int_{\Omega} [L(t,\bx,y_h(t,\bx)) - L(t,\bx,y_{\bar u}(t,\bx))]\, \dx\dt - \int_{\tau}^{\tau+h} \int_{\Omega} L(t,\bx,y_{\bar u}(t,\bx))\,\dx\dt\\
		&= \int_{Q_{\tau+h}} [L(t,\bx,y_h(t,\bx)) - L(t,\bx,y_{\bar u}(t,\bx)) - L_y(t,\bx ,y_{\bar u}(t,\bx))(y_h(t,\bx) - y_{\bar u}(t,\bx))]\, \dx\dt\\
		&\quad+ \int_{Q_{\tau+h}} L_{y}(t,\bx,y_{\bar u}(t,\bx))(y_h(t,\bx)-y_{\bar u}(t,\bx))\,\dx\dt \\
		&\quad- \int_{\tau}^{\tau+h} \int_{\Omega} L(t,\bx,y_{\bar u}(t,\bx))\,\dx\dt\equalscolon I + II + III.
	\end{align*}
    \normalsize
	From Step 1,  $I= o(h)$. Now, for the second term, using integration by parts,
	\begin{align*}
		II&= \int_{\tau+h}^{T}\int_{\Omega} p_{\bar u}(t,\bx)(u_h(t,\bx) - \bar u(t,\bx))\,\dx\dt\\
		&+\int_{\Omega} p_{\bar u}(\tau+h,\bx)\big(y_{h}(\tau+h,\bx) - y_{\bar u}(\tau+h,\bx)\big)\,\dx\\
		&\ge \int_{\Omega} p_{\bar u}(\tau+h,\bx)\big(y_{h}(\tau+h,\bx) - y_{\bar u}(\tau+h,\bx)\big)\,\dx\\
		&= -\int_{\Omega} p_{\bar u}(\tau+h,\bx)\big(y_{\bar u}(\tau+h,\bx)-y_{\bar u}(\tau,\bx)\big)\,\dx.
	\end{align*}
	In the previous estimate, we used the first-order necessary condition of the original problem when proving the inequality.  Now, 
	\begin{align*}
		\frac{\upsilon(\tau+h)-\upsilon(\tau)}{h}&\ge \frac{o(h)}{h}  - \int_{\Omega} p_{\bar u}(\tau+h,\bx)\frac{y_{\bar u}(\tau+h,\bx) - y_{\bar u}(\tau,\bx)}{h}\,\dx\\
		& \hspace{4cm}- \frac{1}{h} \int_{\tau}^{\tau+h}\int_{\Omega} L(t,\bx,y_{\bar u}(t,\bx))\,\dx\dt.
	\end{align*}
	Taking limes inferior on both sides when $h\to0^+$ yields the claim in this step. 
	\smallbreak 
	{Step 3.} 
	We will prove that 
	\begin{align*}
		\limsup_{h\to 0^+} \frac{\upsilon(\tau+h)-\upsilon(\tau)}{h}\le  - \int_{\Omega} p_{\bar u}(\tau,\bx)\,\frac{\rmd^+ y_{\bar u}}{\dt}(\tau,\bx)\,\dx -   \int_\Omega L\big(\tau,\bx,y_{\bar u}(\tau,\bx)\big)\,d\bx.
	\end{align*}
	Let us begin by observing
    \small
	\begin{align*}
		\small 
		& \upsilon(\tau+h) - \upsilon(\tau) \\
		&= \int_{Q_{\tau+h}} \Big[L(t,\bx,y_h(t,\bx)) - L(t,\bx,y_{\bar u}(t,\bx))\Big]\,\dx\dt -\int_{\tau}^{\tau+h} \int_{\Omega} L(t,\bx,y_{\bar u}(t,\bx))\,\dx\dt\\
		&= -\int_{Q_{\tau+h}}\Big[L(t,\bx,y_{\bar u}(t,\bx) - L(t,\bx,y_{h}(t,\bx)) - L_y(t,\bx,y_{h}(t,\bx)(y_{\bar u}(t,\bx) - y_{h}(t,\bx))\Big]\,\dx\dt\\
		&\quad - \int_{Q_{\tau+h}}L_{y}(t,\bx,y_h(t,\bx))(y_{\bar u}(t,\bx)-y_h(t,\bx))\,\dx\dt\\
		&\quad- \int_{\tau}^{\tau+h}\int_{\Omega} L(t,\bx, y_{\bar u}(t,\bx))\,\dx\dt\equalscolon -I - II + III.
	\end{align*}
    \normalsize
	We can use Step 1 to show $I= o(h)$. Now, for the second term, using integration by parts and denoting by $p_h$ the adjoint variable corresponding to $u_h$, 
	\begin{align*}
		II &= \int_{Q_{\tau+h}} p_{h}(t,\bx)(\bar u (t,\bx)- u_h(t,\bx))\,\dx\dt\\
		&\quad+\int_{\Omega} p_{h}(\tau+h,\bx)\big(y_{\bar u}(\tau+h,\bx) - y_{h}(\tau+h,\bx)\big)\,\dx\\
		&\ge \int_{\Omega} p_{h}(\tau+h,\bx)\big(y_{\bar u}(\tau+h,\bx) - y_{h}(\tau+h,\bx)\big)\,\dx\\
		&= \int_{\Omega} p_{h}(\tau+h,\bx)\big(y_{\bar u}(\tau+h,\bx) - y_{\bar u}(\tau,\bx)\big)\,\dx.
	\end{align*}
	We used the first-order necessary condition of problem \cref{Pteta} when proving the inequality. The third term requires no further estimation. Now, 
	\begin{align*}
		\frac{\upsilon(\tau+h)-\upsilon(\tau)}{h} &\le\frac{o(h)}{h}- \int_{\Omega} p_{h}(\tau+h,\bx)\frac{y_{\bar u}(\tau+h,\bx) - y_{\bar u}(\tau,\bx)}{h}\,\dx\\
		& \quad- \frac{1}{h} \int_{\tau}^{\tau+h}\int_{\Omega} L(t,\bx,y_{\bar u}(t,\bx))\,\dx\dt.
	\end{align*}
	Taking the limes superior on both sides when $h\to0^+$ yields the claim in this step proving item (i). 
    \smallbreak\noindent
    We now prove item (ii). Instead of repeating the argument used for item (i), we use Theorem 2.4 to obtain the left derivative directly from differentiability with respect to the initial state. Fix $\tau\in(0,T]$ and set
\[
g(t)\colonequals \upsilon\bigl(t,y_{\bar u}(\tau)\bigr),\qquad t\in[0,\tau].
\]
For $h\in(0,\tau)$, we abbreviate $t_h\colonequals \tau-h$ and
\[
\delta_h \colonequals y_{\bar u}(\tau)-y_{\bar u}(t_h)\in H.
\]
By \cref{thm3}, for every $t_h\in[0,T]$ the map $\eta\mapsto \upsilon(t_h,\eta)$ is Fr\'echet-differentiable at
$\eta=y_{\bar u}(t_h)$ and
\[
\nabla_\eta \upsilon\bigl(t_h,y_{\bar u}(t_h)\bigr)=p_{\bar u}(t_h)\quad\text{in }H.
\]
Hence, applying the differentiability of \cref{thm3} at
$\bigl(t_h,y_{\bar u}(t_h)\bigr)$, there exists a remainder
$r(h)\in\mathbb{R}$ with $r(h)=o(\|\delta_h\|_H)$ such that
\begin{align}
\label{eq:frechet_expansion_timeproof}
\upsilon\bigl(t_h,y_{\bar u}(\tau)\bigr)
=
\upsilon\bigl(t_h,y_{\bar u}(t_h)\bigr)
+\langle p_{\bar u}(t_h),\delta_h\rangle
+r(h).
\end{align}
Since $\frac{d^- y_{\bar u}}{dt}(\tau)$ exists in $H$, we have
$\|\delta_h\|_H
=
\|y_{\bar u}(\tau)-y_{\bar u}(t_h)\|_H
=
O(h)$, and therefore $r(h)=o(h)$. On the other hand, Bellman's optimality principle implies that $\bar u|_{[t_h,T]}$ is a minimizer of
$(\mathbf P_{t_h,y_{\bar u}(t_h)})$ and $\bar u|_{[\tau,T]}$ of
$(\mathbf P_{\tau,y_{\bar u}(\tau)})$; therefore
\begin{align}
\label{eq:bellman_segment}
\upsilon\bigl(t_h,y_{\bar u}(t_h)\bigr)
=\int_{t_h}^{\tau}\int_\Omega L\bigl(t,\bx,y_{\bar u}(t,\bx)\bigr)\,\dx\dt
+\upsilon\bigl(\tau,y_{\bar u}(\tau)\bigr).
\end{align}
Combining \eqref{eq:frechet_expansion_timeproof} and \eqref{eq:bellman_segment} yields
\begin{align*}
\upsilon\bigl(\tau,y_{\bar u}(\tau)\bigr)-\upsilon\bigl(t_h,y_{\bar u}(\tau)\bigr)
&=
-\int_{t_h}^{\tau}\int_\Omega L\bigl(t,\bx,y_{\bar u}(t,\bx)\bigr)\,\dx\dt
-\langle p_{\bar u}(t_h),\delta_h\rangle-r(h).
\end{align*}
Dividing by $h=\tau-t_h$ we obtain
\begin{align}
\label{eq:diffquot_timeproof}
\frac{g(\tau)-g(t_h)}{h}
=
-\frac{1}{h}\int_{t_h}^{\tau}\int_\Omega L\bigl(t,\bx,y_{\bar u}(t,\bx)\bigr)\,\dx\dt
-\Big\langle p_{\bar u}(t_h),\frac{\delta_h}{h}\Big\rangle
-\frac{r(h)}{h}.
\end{align}
Since $y_{\bar u}\in C([0,T];H)$ and $y_Q\in C([0,T];H)$, the map
$t\mapsto \int_\Omega L\bigl(t,\bx,y_{\bar u}(t,\bx)\bigr)\,\dx$ is continuous on $[0,T]$, and therefore
\[
\frac{1}{h}\int_{t_h}^{\tau}\int_\Omega L\bigl(t,\bx,y_{\bar u}(t,\bx)\bigr)\,\dx\dt
\longrightarrow \int_\Omega L\bigl(\tau,\bx,y_{\bar u}(\tau,\bx)\bigr)\,\dx
\quad\text{as}\quad h\longrightarrow0^+.
\]
Moreover, $p_{\bar u}\in W(0,T)\hookrightarrow C([0,T];H)$ implies $p_{\bar u}(t_h)\to p_{\bar u}(\tau)$ in $H$, and the
assumption that $\frac{\rmd^-y_{\bar u}}{\dt}(\tau)$ exists in $H$ gives
\[
\frac{\delta_h}{h}=\frac{y_{\bar u}(\tau)-y_{\bar u}(\tau-h)}{h}\longrightarrow \frac{\rmd^-y_{\bar u}}{\dt}(\tau)\quad\text{as}\quad h\longrightarrow0^+.
\]
Finally, from \eqref{eq:frechet_expansion_timeproof} and the boundedness of $\|\delta_h\|_H/h$ we obtain
\[
\frac{r(h)}{h}=\frac{r(h)}{\|\delta_h\|_H}\,\frac{\|\delta_h\|_H}{h}\longrightarrow  0\quad\text{as}\quad h\longrightarrow0^+.
\]
Passing to the limit in \eqref{eq:diffquot_timeproof} yields the existence of the left derivative of $g$ at $\tau$ and
\[
\frac{\rmd^-}{\dt}\upsilon\bigl(\tau,y_{\bar u}(\tau)\bigr)
=
-\int_\Omega \Bigl[L\bigl(\tau,\bx,y_{\bar u}(\tau,\bx)\bigr)+p_{\bar u}(\tau,\bx)\,\frac{\rmd^-y_{\bar u}}{\dt}(\tau,\bx)\Bigr]\,\dx,
\]
which is the claimed formula in item \emph{(ii)}. This completes the proof.
\end{proof}

	%%%%%%%%%%%%%%%%%
	%\subsection{Proof of \cref{thm4}}
	%%%%%%%%%%%%%%%%%
	%\begin{proof}[Proof of \cref{thm4}]
	%The statement follows by the same arguments as in the %proofs of \cref{thm3} and \cref{thm2}. 
	%\end{proof}
	\iffalse
	We already have established the differentiability      of the map $\eta \to \upsilon(t,y_{\bar u}              (t,\cdot)+\eta)$ for $\eta=0$ for all $t \in [0,T)$. 
	Next, we consider the differentiability of the map $(t,\eta) \to \upsilon(t,\eta)$ at $ (\tau,y_{\bar u}(\tau,\cdot))$.
	For this, we will estimate the term
	\begin{align*}
		\frac{|\upsilon(\tau+h,y_{\bar u}(\tau)+\eta)-\upsilon(\tau,y_{\bar u}(\tau,\cdot))-h\int_\Omega L(\tau,\cdot, \bar y(\tau,\cdot))-\langle \bar p(\tau),\eta\rangle|}{h+\|\eta\|_H}
	\end{align*}
	\fi 
	
	\section{Some general remarks}
	We discuss modifications of the control model covered by the theory and give an outlook on concluding investigations and related results.
	
	\subsection{Differentiability with initial condition in  $H^1(\Omega)$ and $L^\infty(\Omega)$}\label{ss_diff}
	The arguments utilized in the proofs above hold if initial data in $H^1(\Omega)$ is considered. This provides additional regularity for the time derivative of the involved states. We refer to \cite{KB24} for such a control model with an infinite horizon. The most general PDE constraints for the optimal control problem can be considered when the initial data and its perturbation lie in $L^\infty(\Omega)$. In this case, not only do the proofs simplify, but also, due to the boundedness of the involved states, much more general problems can be considered. This allows, for instance, the boundedness assumption of second derivatives of the involved nonlinearities to be replaced by the weaker condition of local boundedness of second derivatives, allowing for wider nonlinearities appearing in the system, for instance, admitting the Schl\"ogl model. 
	An interesting question is whether one can relax the boundedness of second derivatives and still achieve differentiability in \( L^2(\Omega) \).

	\subsection{Differentiability in a neighborhood of the optimal trajectory}\label{diffneigh}
	The differentiability of the value function in a neighborhood of the optimal trajectory can be obtained in certain situations. For this discussion, we restrict ourselves to differentiability with respect to $\eta \in L^\infty(Q)$. We show that there exists $\delta>0$ (independent of $\tau$) such that at each $\tau$, $\frac{\partial \upsilon(\tau, \eta)}{\partial \eta}$ exists for all $\eta$ with $\|\eta-\bar y(\tau)\|_{L^2(Q_\tau)}<\delta$. Following the proofs in the last section, for a minimizer $u_\eta$ of the perturbed problem, having a growth 
	\begin{equation}\label{growthsec5}
		\mathcal J_{\tau,\eta}(u)- \mathcal J_{\tau,\eta}(u_\eta)\geq c\, \|y_u^{\tau, \eta}-y_{u_\eta}^{\tau, \eta}\|_{L^2(Q_\tau)}^2
	\end{equation}
	guarantees the differentiability along the trajectory. Therefore, when assuming that all perturbed problems permit a unique global minimizer with such growth for perturbations $\eta$ sufficiently small, the differentiability of the original problem in a neighborhood of its trajectory is obtained. In some circumstances, we {don't} need to assume the growth of the perturbed problems. For this, we follow the investigation of the stability of a second order sufficient condition in \cite[Section 5.2.1]{JS2024}. We assume that there exists a constant $\tilde c\in \mathbb R_+$ such that the {minimizers} of the original problem satisfies 
	\begin{equation}
		L_{yy}(t,\bx,\bar y(t,\bx))-\bar p(t,\bx)f_{yy}(t,\bx,\bar y(t,\bx)>\tilde c>0.
	\end{equation}
	Then, we find
    \small
	\begin{align*}
		&\mathcal{J}_{\tau,\eta}''(u_\eta)(u-u_\eta)^2\\
		&=\int_\tau^T\int_\Omega(L_{yy}(t,\bx,y_{u_\eta}(t,\bx))-p_{u_\eta}(t,\bx)f_{yy}(t,\bx,y_{u_\eta}(t,\bx)) z^{2;\tau,\eta}_{u_\eta,u-u_\eta}(t,\bx)\,\dx\dt\\
		&=\int_\tau^T\int_\Omega(L_{yy}(t,\bx,\bar y(t,\bx))-\bar p(t,\bx)f_{yy}(t,\bx,\bar y(t,\bx))) z^{2;\tau,\eta}_{u_\eta,u-u_\eta}(t,\bx)\,\dx\dt\\
		&+\int_\tau^T\int_\Omega (L_{yy}(t,\bx,y_{u_\eta}(t,\bx))-L_{yy}(t,\bx,\bar y(t,\bx)))z^{2;\tau,\eta}_{u_\eta,u-u_\eta}(t,\bx)\,\dx\dt\\
		&+\int_\tau^T\int_\Omega (\bar p(t,\bx)-p_{u_\eta}(t,\bx))f_{yy}(t,\bx,\bar y(t,\bx))z^{2;\tau,\eta}_{u_\eta,u-u_\eta}(t,\bx)\,\dx\dt\\
		&+\int_\tau^T\int_\Omega p_{u_\eta}(t,\bx)(f_{yy}(t,\bx,\bar y(t,\bx))-f_{yy}(t,\bx,y_{u_\eta}(t,\bx)))z^{2;\tau,\eta}_{u_\eta,u-u_\eta}(t,\bx)\,\dx\dt\\
		&\geq \tilde c\,\| z^{\tau,\eta}_{u_\eta,u-u_\eta}\|_{L^2(Q_\tau)}^2+ I_1+I_2+I_3.
	\end{align*}
	\normalsize
	The terms $I_j$, $j\in \{1,2,3\}$ can be estimated as follows
	\begin{align*}
		&\vert I_1 \vert \leq \text{Lip}_{L_{yy}} \| y_{u_\eta}-\bar y\|_{L^\infty(Q_\tau)}  \| z^{\tau,\eta}_{u_\eta,u-u_\eta}\|_{L^2(Q_\tau)}^2,\\
		&\vert I_2 \vert \leq \|f_{yy}(\cdot, \bar y)\|_{L^\infty(Q_\tau)} \| \bar p-p_{u_\eta}\|_{L^\infty(Q_\tau)}  \| z^{\tau,\eta}_{u_\eta,u-u_\eta}\|_{L^2(Q_\tau)}^2,\\
		&\vert I_3 \vert \leq \text{Lip}_{f_{yy}} \|p_{u_\eta}\|_{L^\infty(Q_\tau)} \| y_{u_\eta}-\bar y\|_{L^\infty(Q_\tau)}  \| z^{\tau,\eta}_{u_\eta,u-u_\eta}\|_{L^2(Q_\tau)}^2.
	\end{align*}
	Since $L$ is given as an $L^2(Q)$-tracking type objective functional, when the state equation is linear, the terms $I_1$ and $I_3$ vanish, and it is left to estimate
	\begin{align}\label{estinft}
		\|\bar p-p_{u_\eta}\|_{L^\infty(Q_\tau)}&\leq c_r \| \bar y-y_{u_\eta}\|_{L^r(Q_\tau)}\leq c_r \| \bar y-y_{u_\eta}\|_{L^\infty(Q_\tau)}^{\frac{r-2}{r} }\| \bar y-y_{u_\eta}\|_{L^2(Q_\tau)}^{\frac{1}{r}}\nonumber\\
		&\leq c_r c\| \bar y-y_{u_\eta}\|_{L^\infty(Q_\tau)}^{\frac{r-2}{r} }\| \eta\|_{L^2(Q_\tau)}^{\frac{1}{r}}.
	\end{align}
	Thus, the term $I_2$ can be absorbed if the perturbation $\eta$ is sufficiently small. 
	If the state equation is nonlinear, the terms $I_1$ and $I_3$ do not vanish and we need to estimate $ \| y_{u_\eta}-\bar y\|_{L^\infty(Q_\tau)}$. For this, we may either use that if the states are sufficiently close in $\|\cdot\|_{L^ 2(Q)}$, they are so in $\|\cdot\|_{L^ \infty(Q)}$. Alternatively, we may use as an assumption, for given positive constants, $c, \delta$ and $\gamma\in [1,\infty)$ the control growth
	\begin{equation}
		\label{errorboundcont}
		\mathcal J(u) - \mathcal J(\bar u) \ge c\| {u} - \bar u\|_{L^1(Q)}^{1+1/\gamma}
	\end{equation}
	for all $u\in \mathcal{U}$ with $\|u-\bar u\|_{L^ 1(Q)}\leq \delta$.
	The growth \eqref{errorboundcont} implies control stability, and we can estimate the terms $I_1$ and $I_3$ similar to \eqref{estinft}, which allows us to absorb all of the terms and to obtain the stability of the second order sufficient condition. This implies \eqref{growthsec5} and the differentiability of the value function in a neighborhood as claimed.
\subsection{Multiple minimizers}
\label{ss_mm}
%%%%%%%%%%%%
{In the {proofs above} we assumed that the original problem has one unique global minimizer. It is expected from numerical evidence and indicated from theoretical results as in \cite[Theorem 3.9]{DW_2024}, that this is often the case. For affine tracking problems, one can prove that the set of targets that give a unique minimizer is generic in $L^2(\Omega)$. In some tracking problems, the existence of a unique minimizer for every target can be achieved by making natural hypotheses on the data of the problem, see \cite[Corollary 5.3]{CW_2021}. Still, there are also problems with a tracking datum yielding more than one {minimizer} even with regularization terms; see \cite[Theorem 1.1]{P_2023}.} The investigation of what can be said in this case about the differentiability of the value function,  we leave for future work.

\section*{Acknowledgments}
The authors thank the anonymous referees for their careful reading of the manuscript and are grateful for several remarks that helped to improve the presentation of the paper.

\bibliographystyle{siamplain}
\bibliography{references}

\end{document}